\documentclass[12pt]{article}

\usepackage{vmargin}
\usepackage{fancyhdr}
\usepackage{ifthen}
\usepackage{graphicx} 
\usepackage{indentfirst}
\usepackage[english]{babel}

\usepackage{fouriernc}
\DeclareMathAlphabet{\mathcal}{OMS}{cmsy}{m}{n} %otherwise mathcal looks like mathscr

\usepackage{amsmath}   
\usepackage{amssymb}   
\usepackage{amsfonts}
\usepackage{amsthm}
\usepackage{mathrsfs}
\usepackage{pst-node} 
\usepackage{pstricks}

\usepackage{tweaklist}

\usepackage[colorlinks=true,linkcolor=magenta,citecolor=blue]{hyperref}

\usepackage{comment}
\usepackage{color}

\specialcomment{notes}{%
  \begingroup\color{red}\begin{center}\begin{minipage}{12cm} $\blacktriangleright\ $}{%
  \hfill$\ \blacktriangleleft$ \end{minipage} \end{center}\endgroup}
\setpapersize{A4}
\setmarginsrb{33mm}{20mm}{33mm}{25mm}{10mm}{6mm}{0mm}{10mm}
\fancypagestyle{plain}{%
\fancyhf{} % clear all header and footer fields
\lfoot[\thepage]{}
\cfoot[]{}
\rfoot[]{\thepage}
}

\newcommand{\sk}{\smallskip}
\newcommand{\mk}{\medskip}
\newcommand{\bk}{\bigskip}
\renewcommand{\emptyset}{\ensuremath{\varnothing}}     

\newcommand{\G}{\mathbf{G}}
\newcommand{\B}{\mathbf{B}}
\newcommand{\U}{\mathbf{U}}

\newcommand{\T}{\mathbf{T}}
\renewcommand{\P}{\mathbf{P}}
\renewcommand{\L}{\mathbf{L}}

\newlength{\leftlength}
\newlength{\rightlength}
\newlength{\calculskip}
\newcommand{\calculvskip}[1]{%
  \ifthenelse{#1 = 0}{\setlength{\calculskip}{0pt}}{}%
  \ifthenelse{#1 = 1}{\setlength{\calculskip}{\smallskipamount}}{}%
  \ifthenelse{#1 = 2}{\setlength{\calculskip}{\medskipamount}}{}%
  \ifthenelse{#1 = 3}{\setlength{\calculskip}{\bigskipamount}}{}%
  \ifthenelse{#1 = 4}{\setlength{\calculskip}{1cm}}{}%
  \vskip\calculskip
}

\newcommand{\leftcentersright}[4][2]{%
        \settowidth{\leftlength}{#2}%
        \settowidth{\rightlength}{#4}%
        \calculvskip{#1}
        \noindent#2\hskip-\leftlength%
        \hfill#3\hfill
        \mbox{}\hskip-\rightlength#4%
        \vskip\calculskip%
        }

\newcommand{\centers}[2][2]{\leftcentersright[#1]{}{#2}{}}
\newcommand{\leftcenters}[3][2]{\leftcentersright[#1]{#2}{#3}{}}

\makeatletter
\def\svhline{%
  \noalign{\ifnum0=`}\fi\hrule \@height2\arrayrulewidth \futurelet
   \reserved@a\@xhline}
   
\def\hlinewd#1{%
\noalign{\ifnum0=`}\fi\hrule \@height #1 %
\futurelet\reserved@a\@xhline}
 \makeatother

\numberwithin{equation}{section}

\newtheorem{prop}[equation]{Proposition}  

\newtheorem{thm}[equation]{Theorem}
\newtheorem{lem}[equation]{Lemma}

\theoremstyle{definition}

\newtheorem{rmk}[equation]{Remark}

\newcommand{\X}{\mathrm{X}}

\newcommand{\Z}{\mathrm{Z}}
\newcommand{\Hc}{\mathrm{H}_c}

\newcommand{\Rgc}{\mathrm{R}\Gamma_c}

\newcommand{\ol}{\mathop{\otimes}\limits^\mathrm{L}}
\newcommand{\F}{\mathbb{F}}
\newcommand{\qlb}{\overline{\mathbb{Q}}_\ell}

\usepackage{pstricks-add}
\usepackage{lscape}
\usepackage{multirow}
\usepackage{rotating}

\begin{document}

\title{Quotient of Deligne-Lusztig varieties}
\author{Olivier Dudas\footnote{Oxford Mathematical Institute.}
\footnote{The author is supported by the EPSRC, Project No EP/H026568/1 and by Magdalen College, Oxford.}}

\maketitle

\begin{abstract} We study the quotient of parabolic Deligne-Lusztig varieties by a finite unipotent group $\U^F$ where $\U$ is the unipotent radical of a rational parabolic subgroup $\P = \L \U$. We show that in some particular cases the cohomology of this quotient can be expressed in terms of  "smaller" parabolic Deligne-Lusztig varieties associated to the Levi subgroup $\L$.
\end{abstract}

\section*{Introduction}

The very first approach to the representation theory of finite reductive groups is the construction of representations via Harish-Chandra (or parabolic) induction. If $\G$ is a connected reductive group over $\F = \overline{\F}_p$ with an  $\F_q$-structure associated to a Frobenius endomorphism $F : \G \longrightarrow \G$, and $\P$ is an $F$-stable parabolic subgroup with an $F$-stable Levi complement $\L$, one can define, over any ring $\Lambda$, the following functors

\leftcenters{and}{$ \begin{array}[b]{l} \hphantom{{}^*}\mathrm{R}_{\L }^\G \, :  \,  \Lambda \L^F \text{-}\mathrm{mod} \longrightarrow  \Lambda \G^F \text{-}\mathrm{mod}  \\ [5pt]
 {}^*\mathrm{R}_{\L}^\G \, : \, \Lambda \G^F \text{-}\mathrm{mod}  \longrightarrow  \Lambda \L^F \text{-}\mathrm{mod}  \\  \end{array}$}

\noindent called Harish-Chandra induction and restriction functors. One of the main feature of these functors is that they satisfy the so-called Mackey formula: if $\mathbf{Q}$ is another $F$-stable parabolic subgroup with $F$-stable Levi complement $\mathbf{M}$ then

\centers{$  {}^*\mathrm{R}_{\mathbf{M}}^\G \circ  \mathrm{R}_{\L}^\G \ \simeq \ \displaystyle \sum  \, \mathrm{R}_{\L \cap{}^x \mathbf{M}}^\L \circ  {}^*\mathrm{R}_{\L \cap {}^x \mathbf{M}}^{{}^x \mathbf{M}} \circ \mathrm{ad}\, x $}

\noindent where $x$ runs over a explicit finite set associated to $\L$ and $\mathbf{M}$. In addition to being a powerful tool for studying an induced representation, this formula is also essential for proving that the Harish-Chandra functors depend on $\L$ only and not on the choice of  $\P$.

\sk

It turns out that not all the representations of $\G^F$ can be obtained by Harish-Chandra induction (already for $\G = \mathrm{SL}_2(\F)$, many representations are \emph{cuspidal}). To resolve this problem Deligne and Lusztig  defined in \cite{DeLu} a generalised induction in the case where $\P$ is no longer $F$-stable but $\L$ still is. They  constructed morphisms between the Grothendieck groups

\leftcenters{and}{$ \begin{array}[b]{l} \hphantom{{}^*}\mathrm{R}_{\L }^\G \, :  \,  K_0(\Lambda \L^F \text{-}\mathrm{mod}) \longrightarrow  K_0(\Lambda \G^F \text{-}\mathrm{mod})  \\ [5pt]
 {}^*\mathrm{R}_{\L}^\G \, : \, K_0(\Lambda \G^F \text{-}\mathrm{mod})  \longrightarrow  K_0(\Lambda \L^F \text{-}\mathrm{mod})  \\  \end{array}$}

\noindent still satisfying the Mackey formula. These morphisms come from a virtual character given by the $\ell$-adic cohomology of a quasi-projective variety $\widetilde \X_{\L,\P}$, the \emph{parabolic Deligne-Lusztig variety} associated to $(\L,\P)$. Here, $\Lambda$ is a finite extension of  $\mathbb{Q}_\ell$, $\mathbb{Z}_\ell$ or $\mathbb{\F}_\ell$. 

\sk

When $\Lambda$ is a finite extension of $\mathbb{Q}_\ell$ the category $\Lambda \G^F$-$\mathrm{mod}$ is semisimple, and its Grothendieck group encodes most of the information. However, in the modular framework, that is when $\Lambda = \mathbb{Z}_\ell$ or $\F_\ell$, the Deligne-Lusztig induction and restriction morphisms give only partial information on the category of modules. To obtain homological properties, one needs to consider the complex $\Rgc(\X,\Lambda)$ representing the cohomology of the variety in the derived category $D^b(\Lambda \G^F$-$\mathrm{mod})$. Using this point of view, Bonnafé and Rouquier  defined in \cite{BR1} triangulated functors

\leftcenters{and}{$ \begin{array}[b]{l} \hphantom{{}^*}\mathcal{R}_{\L \subset \P }^\G \, :  \,  D^b(\Lambda \L^F \text{-}\mathrm{mod}) \longrightarrow D^b(\Lambda \G^F \text{-}\mathrm{mod})  \\ [5pt]
 {}^*\mathcal{R}_{\L \subset \P }^\G \, : \, D^b(\Lambda \G^F \text{-}\mathrm{mod})  \longrightarrow D^b(\Lambda \L^F \text{-}\mathrm{mod}).  \\  \end{array}$}

\noindent Unlike the previous functors, these are not expected to satisfy a naive Mackey formula as they highly depend on the choice of $\P$. However, there is a good evidence that the composition ${}^*\mathcal{R}_{\mathbf{M}\subset \mathbf{Q}}^\G \circ  \mathcal{R}_{\L \subset \P}^\G$ should be somehow related to functors associated to smaller Levi subgroups. The purpose of this paper is to investigate the case where $\mathbf{Q}$ is $F$-stable. If $\U$ denotes its unipotent radical, then the composition ${}^*\mathcal{R}_{\mathbf{M}\subset \mathbf{Q}}^\G \circ  \mathcal{R}_{\L \subset \P}^\G$ is induced by the cohomology of the quotient variety $\U^F \backslash \widetilde \X_{\L,\P}$.

\sk

In the original paper of Deligne and Lusztig \cite{DeLu}, the Levi subgroup $\L$ is a torus and $ \widetilde \X_{\L,\P}$ corresponds to some element $w$ of the Weyl group $W$ of $\G$. The motivating example is when $(\L,\P)$ represents a Coxeter torus, that is when $w$ is a Coxeter element of $W$. In that case,  the variety  $\X_{\L,\P} = \widetilde \X_{\L,\P}/\L^F$ is contained in the maximal Schubert cell and its quotient by $\U^F$ has been computed by Lusztig in \cite{Lu}. In the case where $\Lambda = \qlb$ it is given by the following quasi-isomorphism of $\mathbf{M}^F$-modules:

\centers{$\Rgc(\U^F \backslash \X_{\L,\P},\qlb) \, \simeq \,  \Rgc(\X_{\L\cap \mathbf{M}, \P \cap \mathbf{M}}, \qlb)  \otimes \Rgc((\F^\times)^d,\qlb)$}

\noindent where $d$ is the semisimple $\F_q$-index of $\mathbf{M}$ in $\G$. Surprisingly, this isomorphism does not come from a $\mathbf{M}^F$-equivariant isomorphism of varieties, and we will see that it is more natural to study the quotient of $\widetilde \X_{\L,\P}$ instead of $\X_{\L,\P}$.

\sk

In general, the variety $\X_{\L,\P}$ is not contained in only one Schubert cell. The strategy towards the determination of the cohomology of $\U^F \backslash  \widetilde \X_{\L,\P}$ will consist in the following steps:

\begin{itemize}

\item decompose the variety $\widetilde \X_{\L,\P}$ into \emph{pieces} $\widetilde \X_x$ coming from the decomposition of $\G/\P$ into $\mathbf{Q}$-orbits (see Section \ref{sec2});

\item in some well-identified cases, express the cohomology of $\U^F \backslash \X_x$ in terms of parabolic Deligne-Lusztig varieties associated to Levi subgroups of $\mathbf{M}$ (see Section \ref{sec3}). 

\end{itemize}

\noindent The second step is undoubtedly the most difficult. We are able to provide a satisfactory solution to this problem in presumably very specific situations, namely when the pair $(\L \cap {}^x \mathbf{M},\P \cap {}^x \mathbf{M})$ is close to $(\L,\P)$ (see Theorem \ref{mainthm} for more details).  However, it turns out that our main result is general enough to cover most of the Deligne-Lusztig varieties associated to unipotent $\Phi_d$-blocks with cyclic defect group. This should give many new results on the geometric version of Broué's abelian defect conjecture. To illustrate this phenomenon, we compute in Section \ref{sec33} the principal part of the cohomology of the parabolic variety associated to the principal $\Phi_{2n-2}$-block  for a group of type $B_n$ as well as its Alvis-Curtis dual. In subsequence papers this
baby example will be supplemented by the following more involved results:

\begin{itemize}

\item for exceptional groups, the determination of the cohomology of varieties associated to principal $\Phi_d$-blocks when $d$ is the largest regular number besides the Coxeter number. This should be enriched with predictions for the corresponding Brauer trees;

\item for groups of type $A_n$, the determination of the cohomology of varieties associated to any unipotent block from the knowledge of the cohomology of the variety $\X(\mathbf{w}_0^2)$. 

\end{itemize}

%The advantage of our approach is that we obtain results for the cohomology of $\X_{\L,\P}$ with coefficients in any unipotent local system (see Remark \ref{rmkunip}).

\sk

%However, it turns out that it covers many interesting cases,  

\section{Parabolic Deligne-Lusztig varieties\label{sec1}}

Let $\G$ be a connected reductive algebraic group, together with an isogeny $F$, some power of which is a Frobenius endomorphism. In other words, there exists a positive integer $\delta$ such that $F^\delta$ defines a split $\mathbb{F}_{q^\delta}$-structure on $\G$ for a certain power $q^\delta$ of the characteristic $p$ (note that $q$ might not be an integer). For all  $F$-stable algebraic subgroup $\mathbf{H}$ of $\G$, we will denote by $H$ the finite group of fixed points $\mathbf{H}^F$.  \sk

We fix a Borel subgroup $\B$ containing a maximal torus $\T$ of $\G$ such that both $\B$ and $\T$ are $F$-stable. They define a root sytem $\Phi$ with basis $\Delta$, and a set of positive (resp. negative) roots $\Phi^+$ (resp. $\Phi^-$). Note that the corresponding Weyl group $W$ is endowed with an action of $F$, compatible with the isomorphism $W \simeq N_\G(\T)/\T$.  The set of simple reflections will be denoted by $S$. We shall also consider representatives $\{\dot w \, | \, w \in W\}$ of $W$ in $N_\G(\T)$ compatible with the action of $F$ (this is possible by \cite[Proposition 8.21]{DM3}).

\sk

To any subset $I \subset S$ one can associate a standard parabolic subgroup $\P_I$ containing $\B$ and a standard Levi subgroup $\L_I$ containing $\T$. If $\U_I$ denotes the unipotent radical of $\P_I$, the parabolic subgroup can be written as $\P_I = \L_I \U_I$. Let $\U$ (resp. $\U^-$) be the unipotent radical of $\B$ (resp. the opposite Borel subgroup $\B^-$). Each root $\alpha$ defines a one-parameter subgroup $\U_\alpha$, and we will denote by $u_\alpha : \mathbb{F} \longrightarrow \U_\alpha$ an isomorphism of algebraic group. In order to simplify the calculations, we shall choose these isomorphisms so that $u_\alpha(\lambda) \, \dot s_\alpha = u_{-\alpha}(\lambda^{-1}) \alpha^\vee(\lambda) u_\alpha(-\lambda^{-1})$.  Note that the groups $\U_\alpha$ might not be $F$-stable in general even though the groups $\U$ and $\U^-$ are. 

\sk

Finally, we denote by $B_W^+$ (resp $B_W$) the Artin-Tits monoid (resp. Artin-Tits group) of $W$, and by $\mathbf{S} = \{\mathbf{s}_\alpha \, | \, \alpha \in \Delta\}$ its generating set. The reduced elements of $B_W^+$ form a set $\mathbf{W}$ which is in bijection with $W$ via the canonical projection $B_W \twoheadrightarrow W$. We shall also consider the semi-direct product $B_W \rtimes \langle F \rangle$ where $F \cdot \mathbf{b} = {}^F \mathbf{b} \cdot F$.

\sk

Let $\mathbf{I}$ be a subset of $\mathbf{S}$ and denote by $B_\mathbf{I}^+$ the submonoid of $B_W^+$ generated by $\mathbf{I}$. Following \cite{DM3}, we will denote by $\mathbf{I} \mathop{\longrightarrow}\limits^\mathbf{b} {}^F \mathbf{I}$ any pair $(\mathbf{I},\mathbf{b})$ with  $\mathbf{b} \in B_W^+$ satisfying the following properties:

\begin{itemize}

\item any left divisor of $\mathbf{b}$  in $B_\mathbf{I}^+$ is trivial;

\item ${}^{\mathbf{b} F} \mathbf{I} = \mathbf{I}$, that is every $\mathbf{s} \in \mathbf{I}$ satisfies $\mathbf{b}^{-1} \mathbf{s} \mathbf{b} \in {}^F \mathbf{I}$.

\end{itemize} 

\noindent Digne and Michel have constructed in \cite{DM3} a \emph{parabolic Deligne-Lusztig variety} $\X(\mathbf{I},\mathbf{b} F)$ associated to any such pair. Note that when $\mathbf{b} = \mathbf{w} \in \mathbf{W}$ and if $w$ denotes its image by the canonical projection $B_W \twoheadrightarrow W$, the previous conditions are equivalent to $w$ being $I$-reduced and ${}^{wF} I =I$. In that case, the variety $\X(\mathbf{I},\mathbf{w} F)$ can be written

\centers{$ \X(I,wF) = \big\{g \in \G\, \big| \, g^{-1} \, {}^F g \in \P_I w \, {}^F \P_I\big\} \big/\P_I .$}

\noindent As in the case of tori, we can construct a Galois covering of $\X(I,wF)$. It is well-defined up to a choice of a representative $n$ of $w$ in $N_\G(\T)$:

\centers{$ \widetilde \X(I,nF) = \big\{g \in \G \, \big| \, g^{-1} \, {}^F g \in \U_I n \, {}^F \U_I\big\}\big/\U_I$.}

\noindent  The natural projection $\G/\U_I \longrightarrow \G/\P_I$ makes $\widetilde \X(I,nF)$ a $\L_I^{n F}$-torsor over $\X(I,wF)$. By using an $F$-stable Tits homomorphism $t: B_W \longrightarrow N_\G(\T)$ extending $w \in W \longmapsto \dot w$, Digne and Michel have generalised in \cite{DM3} this construction to any element $\mathbf{I} \mathop{\longrightarrow}\limits^\mathbf{b} {}^F \mathbf{I}$. The corresponding variety will be denoted by $\widetilde \X(\mathbf{I}, {\mathbf{b}} F)$. It is a $\L_I^{t(\mathbf{b}) F}$-torsor over $\X(\mathbf{I},\mathbf{b}F)$. When $\mathbf{b} = \mathbf{w} \in \mathbf{W}$ we shall simply denote $t(\mathbf{w})$ by $\dot{{w}}$.

\begin{rmk} When $\mathbf{I}$ is empty, we obtain the usual Deligne-Lusztig varieties $\X(\mathbf{b}F)$ and $\widetilde \X({\mathbf{b}} F)$ associated to any element $\mathbf{b}$ of the Braid monoid (as defined in \cite{BMi2} or \cite{BR1}).
\end{rmk}

\section{\label{sec2}Decomposing the quotient of \texorpdfstring{$\X(I,wF)$}{X(I,wF)}}

Let $(I,w)$ be a pair consisting of an element $w$ of $W$ and a subset $I$ of $S$ such that $w$ is $I$-reduced and ${}^{wF} I = I$. Let $J$ be another subset of $S$. If $J$ is $F$-stable, then so is the corresponding standard parabolic subgroup $\P_J$ and its unipotent radical $\U_J$. In this section we are interested in describing  the quotient of the parabolic Deligne-Lusztig variety

\centers{$ \X(I,wF) \, = \, \big\{g \in \G\, | \, g^{-1} \, {}^F g \in \P_I w \, {}^F \P_I \big\}\big/\mathbf{P}_I $}

\noindent by the finite unipotent group $U_J$. Our main goal is to express this quotient   (or at least its cohomology) in terms of "smaller parabolic varieties" associated to the Levi subgroup $\L_J$.

\sk

Throughout this paper, $\Lambda$ will be any extension of the ring $\mathbb{Z}_\ell$ of $\ell$-adic integers. We shall always assume that $\ell$ is different from $p$, so that by cohomology over $\Lambda$ we mean the extension of the étale cohomology of quasi-projective varieties with coefficients in $\mathbb{Z}_\ell$. The properties of  $\Rgc(-,\Lambda)$ that we will use are either classical or can be found in \cite{Rou}.

\subsection{A general method\label{sec21}}

Recall that the partial flag variety $\G/\mathbf{P}_I$ admits a decomposition into $\P_J$-orbits $\G/\P_I = \coprod \P_J x \P_I$ where $x$ runs over any set of representatives of $W_J \backslash W / W_I$. The restriction of this decomposition to $\X(I,wF)$ can be written as 
\begin{equation}
  \X(I,wF) \, =  \, \displaystyle \coprod_{x \in [W_J \backslash W / W_I]} \big\{px\mathbf{P}_I \in \P_J x \P_I/\mathbf{P}_I \, \big| \, p^{-1} \, {}^F p  \in x(\P_I w \, {}^F \P_I )\, {}^F x ^{-1}\big\}.
\end{equation}
\noindent We will denote by $\X_x = \X(I,wF) \cap \P_J x \P_I$ a piece of this decomposition. It is a locally closed $P_J$-subvariety of $\X(I,wF)$. Now, each of these pieces can be lifted up to $\P_J$. More precisely, if we define the variety

\centers{$ \Z_x \, = \, \big\{p \in \P_J \, \big| \, p^{-1} \, {}^F p  \in x(\P_I w \, {}^F \P_I ) \, {}^F x^{-1}\big\} $}

\noindent then the canonical projection $\G \longrightarrow \G/\P_I$ induces a fibration $\Z_x \longrightarrow \X_x$ with fiber isomorphic to $\P_J \cap {}^x \P_I$. Now if we assume that $x$ is $J$-reduced-$I$, the intersection $\P_J \cap {}^x \P_I$ can be decomposed as $\P_J \cap {}^x \P_I = (\L_J \cap {}^x \P_I) \cdot (\U_J \cap {}^x \U)$. Furthermore, $\L_J \cap {}^x \P_I$ is a standard parabolic subgroup of $\L_J$ (it contains $\L_J \cap \B$) and hence it can be written $\L_J \cap \P_{K_x}$ with $K_x = J \cap {}^x \Phi_I$. The cohomology of $\X_x$ is thus given by
\begin{equation}\label{xxeq}
 \Rgc(\X_x,\Lambda) \simeq \Rgc(\Z_x / \L_J \cap \P_{K_x},\Lambda) [2 \dim \U_J \cap {}^x \U].
 \end{equation}
The advantage of this description is that the quotient of the variety $\Z_x$ by $U_J$ is easier to compute. If we decompose $p \in \P_J$ as $p= ul \in \U_J \L_J$ then the quotient variety can be written (see for example \cite[Proposition 1.3]{Du1})

\centers{$ U_J \backslash \Z_x \, = \, \big\{(\bar p,l) \in \big[ (x\P_I w \, {}^F \P_I  \, {}^F x ^{-1}) \cap \P_J\big] \times \L_J \, \big| \, \pi_J(\bar p) = l^{-1} \, {}^F l \big\}$}

\noindent where $\pi_J : \P_J \longrightarrow \L_J $ is the canonical projection. 

\sk

Our aim is to relate this variety to "smaller" parabolic Deligne-Lusztig varieties. For that purpose, we need to identify the double cosets in which $l^{-1} {}^F l$ lies, which amounts to decomposing the intersection $(x\P_I w \, {}^F \P_I  \, {}^F x^{-1}) \cap \P_J$ as well as its image under $\pi_J$. Let $v \in W_J$ be a $K_x$-reduced-${}^F K_x$ element. We can decompose the double coset $\P_{K_x} v \, {}^F \P_{K_x}$ as follows:

\centers{$ \P_{K_x} v \, {}^F \P_{K_x}  = (\L_J \cap \P_{K_x})\, \U_J v \, ( \L_J \cap \, {}^F \P_{K_x}).$}

\noindent Since $\L_J\cap \P_{K_x} = \L_J \cap {}^x \P_I$ is a subgroup of  ${}^x \P_I$, the intersection $(x\P_I w \, {}^F \P_I  \, {}^F x ^{-1}) \cap (\P_{K_x} v \, {}^F \P_{K_x})$ is non-empty if and only if $(x\P_I w \, {}^F \P_I  {}^F x^{-1} v^{-1}) \cap \U_J$ is. In this case, the projection $\pi_J : \P_J \longrightarrow \L_J$ induces a fibration $(x\P_I w \, {}^F \P_I  \, {}^F x ^{-1}) \cap (\P_{K_x} v \, {}^F \P_{K_x}) \longrightarrow  (\L_J \cap \P_{K_x}) \, v\, ( \L_J \cap \, {}^F \P_{K_x})$ with fiber isomorphic to $(x\P_I w \, {}^F \P_I  {}^F x^{-1} v^{-1}) \cap \U_J$. If we define $\Z_{x}^v$ to be the variety

\centers{$\Z_{x}^v \, = \, \big\{(\bar p,l) \in \big[ (x\P_I w \, {}^F \P_I \, {}^F x^{-1}) \cap (\P_{K_x} v \, {}^F \P_{K_x}) \big] \times \L_J \, \big| \, \pi_J(\bar p) = l^{-1} \, {}^F l \big\}$}

\noindent then we obtain a decomposition of $U_J \backslash \Z_x$ into locally closed subvarieties together with $L_J$-equivariant maps
\begin{equation}
 \Z_{x}^v  \longrightarrow \big\{ l \in \L_J  \, | \, l^{-1} \, {}^F l \in \L_J \cap \P_{K_x} v \, {}^F (\L_J \cap \P_{K_x}) \big\} \label{surjmor}
\end{equation}
\noindent with fibers isomorphic to $(x\P_I w \, {}^F \P_I  \, {}^F x^{-1} v^{-1}) \cap \U_J$. 

\begin{rmk} In the case where ${}^{vF} K_x = K_x$, the quotient by $\L_J \cap \P_{K_x}$ of the variety on the right-hand side of \ref{surjmor} can be identified with the parabolic Deligne-Lusztig variety associated to $K_x \mathop{\longrightarrow}\limits^v {}^F K_x$. We shall, by convenient abuse of notation, denote it by $\X_{\L_J}(K_x,vF)$ even when $vF$ does not normalise $K_x$.
\end{rmk}

Finally, we set $\mathfrak{Z}_x^v = \Z_x^ v/\L_J\cap \P_{K_x}$. The right action of $\U_J \cap {}^x \U$ on $\Z_x$  induces an action by $F$-conjugation on $\mathfrak{Z}_x^v$ and let  $\mathfrak{X}_x^v = \mathfrak{Z}_x^v / \U_J \cap {}^x \U_I$ be the quotient (equivalently, it is the image of $\Z_x^v$ by the morphism $U_J \backslash \Z_x \twoheadrightarrow U_J \backslash \X_x$). At this point we have obtained

\begin{itemize}

\item A decomposition of $U_J \backslash \X(I,wF)$ into some locally closed $L_J$-varieties $\mathfrak{X}_x^v$.

\item A quasi-isomorphism $\Rgc(\mathfrak{X}_x^v,\Lambda) \simeq \Rgc(\mathfrak{Z}_x^v,\Lambda) [2 \dim \U_J \cap {}^x \U]$ (obtained as in \ref{xxeq}).

\item A $L_I$-equivariant morphism $\mathfrak{Z}_x^v \longrightarrow \X_{L_J}(K_x,vF)$ with fiber isomorphic to $(x\P_I w \, {}^F \P_I  \, {}^F x^{-1} v^{-1}) \cap \U_J$.

\end{itemize}

\noindent Therefore, if we want to express the cohomology of $U_J \backslash \X(I,wF)$ in terms of the different varieties $\X_{L_J}(K_x,vF) $ that can appear we need to refine the description of the latter morphism. This will be done in Section \ref{sec23}  after discussing the case of parabolic varieties associated to elements of the Braid monoid. 

\begin{rmk} When ${}^{vF} K_x = K_x$,  we can  actually be more precise: $l^{-1} \, {}^F l $ can be written uniquely as $l_1 \dot v \, {}^F l_2$ with $ l_1 \in (\L_J \cap  \U_{K_x}) \cap {}^{vF} (\L_J \cap  \U_{K_x}^-) $ and $l_2 \in \L_J \cap \P_{K_x}$. Then for $z \in (x\P_I w \, {}^F \P_I  \, {}^F x^{-1} v^{-1}) \cap \U_J$  we have $( l_1 z \dot v {}^F l_2,l) \in \Z_x^v$ and  all the elements are obtained that way. In other words, we have the following isomorphism of varieties

\centers{$ \Z_x^v \, \simeq \, \big[(x\P_I w \, {}^F \P_I  \, {}^F x^{-1} v^{-1}) \cap \U_J \big] \times \big\{ l \in \L_J  \, | \, l^{-1} \, {}^F l \in \L_J \cap \P_{K_x} v \, {}^F (\L_J \cap \P_{K_x}) \big\}. $}

\noindent Through this isomorphism the group $L_J$ (resp. $\L_J \cap \P_{K_x}$) acts on $l \in \L_J$ by left (resp. right) multiplication. However, it is more difficult to describe the action of $\L_J \cap \P_{K_x}$ on $(x\P_I w \, {}^F \P_I  \, {}^F x^{-1} v^{-1}) \cap \U_J$. In particular, $\mathfrak{Z}_x^v$ is in general not isomorphic to $\big[(x\P_I w \, {}^F \P_I  \, {}^F x^{-1} v^{-1}) \cap \U_J\big] \times \X_{\L_J}(K_x,vF)$. We shall  nevertheless  give many examples where the cohomology of these two varieties coincide.

\end{rmk}
\subsection{Elements of the Braid monoid}

By \cite[Section 6]{DM3} any element $\mathbf{I} \mathop{\longrightarrow}\limits^\mathbf{b} {}^F \mathbf{I}$ can be decomposed as $\mathbf{I} = \mathbf{I}_1 \mathop{\longrightarrow}\limits^{\mathbf{w}_1}$ $ \mathbf{I}_2  \mathop{\longrightarrow}\limits^{\mathbf{w}_2} \cdots$  $\mathop{\longrightarrow}\limits^{\mathbf{w}_r} \mathbf{I}_{r+1}= {}^F \mathbf{I}$ where $\mathbf{w}_i \in \mathbf{W}$. Using this property one can easily generalize the previous constructions to $\X( \mathbf{I},\mathbf{b}F)$: to each tuple $\mathbf{x} = (x_1, \ldots, x_r)$ with $x_i$ a $J$-reduced-${I}_i$ element of $W$ one can associate varieties $\X_\mathbf{x}$ and $\Z_\mathbf{x}$ such that

\centers{$\Z_\mathbf{x} \, = \, \left\{ (p_1,\ldots,p_r) \in (\P_{J})^r\, \left| \, \begin{array}{l} p_i^{-1} p_{i+1} \in x_i \P_{I_i} w_i \, \P_{I_{i+1}} x_{i+1}^{-1} \\[4pt]   p_r^{-1}\, {}^F p_1 \in x_r \P_{I_r} w_r \, {}^F \P_{I_{1}} \, {}^F x_{1}^{-1} \end{array}\right.\right\}$ }

\leftcenters{and}{$ \Rgc(\X_\mathbf{x},\Lambda) \,\simeq \, \Rgc\big(\Z_\mathbf{x} \big/ \prod \L_J \cap \P_{K_{x_i}},\Lambda\big) \big[2 \sum \dim \U_J \cap {}^{x_i} \U\big]$}

\noindent with $K_{x_i} = J \cap {}^{x_i} \Phi_{I_i}$.

\sk

 By looking at the intersections of $x_i \P_{I_i} w_i \P_{I_{i+1}} x_{i+1}^{-1}$ with double cosets of the form  $\P_{K_{x_i}} v_i \P_{K_{x_{i+1}}}$ one can decompose $U_J \backslash \Z_{\mathbf{x}}$ into locally closed subvarieties $\Z_\mathbf{x}^{\mathbf{v}}$ together with $L_J$-equivariant maps 
\begin{equation}\label{eqquotient} \Z_{\mathbf{x}}^\mathbf{v} \, \longrightarrow \,
 \left\{ (l_1, \ldots,l_r)  \in (\L_J)^r  \, \left|  \, \begin{array}{r} l_i^{-1} \, l_{i+1} \in \big(\L_J \cap \P_{K_{x_i}}\big)\, v_i \, \big(\L_J \cap \P_{K_{x_{i+1}}}\big) \\[5pt] 
l_r^{-1} \, {}^F l_{1} \in \big(\L_J \cap \P_{K_{x_r}}\big) \, v_r\, {}^F \big(\L_J \cap \P_{K_{x_{1}}}\big) 
\end{array} \right.\right\} 
\end{equation}
\noindent with fibers isomorphic to

 \centers{$ \U_J \cap \big(x_r\P_{I_r} w_r \, {}^F \P_{I_1}  \, {}^F x_1^{-1} v_r^{-1}\big) \times \displaystyle \, \prod_{i=1}^{r-1} \, \U_J \cap \big(x_i\P_{I_i} w_i  \P_{I_{i+1}}  x_{i+1}^{-1} v_i^{-1}\big). $}

 \noindent In the case where ${}^{v_i} K_{x_{i+1}} = K_{x_i}$ and ${}^{v_{r} F}K_{x_1} = K_{x_r}$, the quotient by $\prod \L_J \cap \P_{K_{x_i}}$ of  the variety on the right-hand side of \ref{eqquotient} can be identified with the parabolic Deligne-Lusztig variety $\X_{\L_J}(\mathbf{K}_{x_1},\mathbf{v}_1 \cdots \mathbf{v}_r F)$. 

\subsection{A further decomposition\label{sec23}}

We now study the intersection $(x\P_I w \, {}^F \P_I  \, {}^F x^{-1} v^{-1}) \cap \U_J$ in order to obtain information on the morphism $\mathfrak{Z}_x^v \longrightarrow \X_{L_J}(K_x,vF)$ defined at the end of Section \ref{sec21}. This will be achieved using the Curtis-Deodhar decomposition.

\sk

Let $x,w,w'$ be elements of $W$, and fix a reduced expression $w=s_1 \cdots s_r$ of $w$. Recall that a \emph{subexpression} of $w$ (with respect to the decomposition $w=s_1 \cdots s_r$) is an element of $\Gamma = \{1,s_1\} \times \cdots \times \{1,s_r\}$. Such a subexpression $\gamma = (\gamma_1,\ldots,\gamma_r)$ is said to be \emph{$x$-distinguished} if  $\gamma_i = s_i$ whenever $x \gamma_1 \cdots \gamma_{i-1} s_i > x \gamma_1 \cdots \gamma_{i-1}$. The main result in \cite{Deo} and \cite{Cur} gives a  decomposition of the double  Schubert cell $\B w \B \cap (\B)^x\, w' \B \, \subset \G/\B$ in terms of certain $x$-distinguished subexpressions of $w$, as well as an explicit parametrisation of each piece (see \cite[Section 2.2]{Du2} for more details).

\begin{thm}[Deodhar, Curtis]\label{deodec} Let $w,w',x$ be elements of the Weyl group and $w=s_1 \cdots s_r$ be a reduced expression of $w$. There exists a decomposition of $\B w \B \cap (\B)^x\, w' \B$ into locally closed subvarieties

\centers{$\B w \B \cap (\B)^x\, w' \B \, = \, \displaystyle \coprod_{\gamma \in \Gamma_{w'}} \Omega_\gamma w' \B$} 

\noindent where $\gamma$ runs over the set $\Gamma_{w'}$ of  subexpression of $w$ whose product is $w'$. Furthermore, the decomposition has the following properties:

\begin{itemize}

\item[$\mathrm{(i)}$] Each cell $\Omega_\gamma w' \B$ is stable by multiplication by $\U\cap\U^x$;

\item[$\mathrm{(ii)}$] $\Omega_\gamma \subset \U^x$ and the restriction of the map $\B^x \longrightarrow (\B)^x w' \B/\B$ to $\Omega_\gamma$ is injective;

\item[$\mathrm{(iii)}$] $\Omega_\gamma$ is non-empty if and only if $\gamma$ is $x$-distinguished;

\item[$\mathrm{(iv)}$]  If $\Omega_\gamma$ is non-empty, then it is isomorphic to $\mathbb{A}_{n_\gamma} \times (\mathbb{G}_m)^{m_\gamma}$ where 

\centers{$n_\gamma = \#\{i=1,\ldots,r \, | \, x \gamma_1 \cdots \gamma_{i-1} s_i > x \gamma_1 \cdots \gamma_{i-1}\}$}

\leftcenters{and}{$m_\gamma = \#\{i = 1,\ldots,r \, | \, \gamma_i = 1\}. $}

\end{itemize}

\end{thm}

\begin{rmk} For convenience, we will always denote by $\mathbb{G}_m$ the spectrum of the ring $\F[t,t^{-1}]$ although 
we will not necessarily use its group structure.
\end{rmk}

 In order to use this result, we first write the fiber of \ref{surjmor} as

\centers{$ (x\P_I w \, {}^F \P_I  \, {}^F x^{-1} v^{-1}) \cap \U_J \, = \,( x\B W_I w \B \, {}^F x^{-1} v^{-1}) \cap \U_J$.}

\noindent Let $y \in W_I$, and let $\gamma$ be a $x$-distinguished subexpression of $yw$ whose product is $w'=x^{-1}  v \, {}^F v$. Then the map $(z,z') \in \Omega_\gamma \times \U^x \cap {}^{w'} \U \longmapsto zz'w' \in \B w \B \cap (\U)^x w'$ is well-defined and it is injective by Theorem \ref{deodec}.(ii).  By taking the union over such subexpressions, we obtain the following decomposition

\centers{$ \U \cap x\B yw \B  (v\, {}^F x )^{-1} \, = \, \displaystyle \hskip-3mm \bigsqcup_{\gamma \in \Gamma_{x^{-1}v\, {}^Fx}} \hskip-3mm \big({}^x \Omega_\gamma \big)\cdot \big(\U \cap {}^{v\, {}^Fx} \U \big).$}

\noindent Note that we do not need to fix a reduced expression of $y$: indeed, since $x$ is reduced-$I$, the subexpression $\gamma$ will start with any reduced expression of $y$. 

\sk

Furthermore, by Theorem \ref{deodec}.(i), each coset $\Omega_\gamma x^{-1}v\, {}^F x \, \B$ is stable by left-multiplication by $\U \cap \U^x$, and therefore all the varieties occurring in the previous decomposition are stable by the left action of ${}^x \U \cap \U$. Since $x$ is $J$-reduced, they are  in particular  stable by the action of $\L_J\cap \U$. Taking the image by the projection $\varpi_J : \U \longrightarrow \U_J$ associated to the decomposition $\U = (\U \cap \L_J) \, \U_J$ we obtain
\begin{equation}\label{gcell} \U_J \cap x\B yw \B (v\, {}^F x)^{-1} \, = \, \displaystyle \hskip-3mm \bigsqcup_{\gamma \in \Gamma_{x^{ -1}v \, {}^Fx}} \hskip-3mm \varpi_J \big({}^x \Omega_\gamma \cdot (\U \cap {}^{v\, {}^Fx} \U ) \big) \, = \, \displaystyle \hskip-3mm \bigsqcup_{\gamma \in \Gamma_{x^{ -1}v \, {}^Fx}} \hskip-3mm \Upsilon_\gamma.
\end{equation}
\noindent In many interesting examples, the intersection  $(x\P_I w \, {}^F \P_I  \, {}^F x^{-1} v^{-1}) \cap \U_J$ will always consist of at most one cell $\Upsilon_\gamma$, which will be isomorphic to $(\mathbb{G}_m)^r \times \mathbb{A}_s$ for some integers $r,s$. Note that in this case, the cell is automatically stable by the action of $(\L_J \cap \P_{K_x} )\cap {}^{v F} (\L_J \cap \P_{K_x})$ by conjugation. If in addition one can find an equivariant embedding $\Upsilon_\gamma \subset \mathbb{A}_{r+s}$, then the cohomology of $U_J \backslash \X_x$ can be obtained by shifts of the cohomology of $\X_{L_J}(K_x,vF)$. We shall not make this claim more precise as we will  encounter only the cases where $r=0$ or $1$.

\begin{rmk}\label{rmktest}The decomposition \ref{gcell} gives a combinatorial test for the emptyness of a piece $\X_x$: it is non-empty if and only if there exist $y \in W_I$ and an $x$-distinguished subexpression $\gamma$ of $yw$ such that the product of the elements of $\gamma$ lies in $x^{-1} W_J \, {}^F x$.
\end{rmk}

\subsection{Examples\label{sec24}}

In this section we give examples for which the previous method is effective. Some of them will nevertherless suggest that one should rather work with the variety $\widetilde \X(I,\dot w F)$ instead of $\X(I,wF)$. 

\mk

\noindent \thesubsection.1. \textbf{Fibers are affine spaces.} Let $J$ be an $F$-stable subset of $S$. Assume that there exists a $J$-reduced-$I$ element $x$ such that $xw\ {}^Fx^{-1} \in W_J$ and let $v$ be the corresponding $W_{K_x}$-reduced element. Then we have $U_J \backslash \Z_x = \Z_{x}^v$ and the map 

\centers{$ \mathfrak{Z}_x^v  \twoheadrightarrow \X_{\L_J}(K_x,vF)$}

\noindent has affine fibers. In particular, the cohomology of the varieties $U_J \backslash \X_x$ and $ \X_{\L_J}(K_x,vF)$ differ only by a shift.

\sk

 Let  $v' \in W_J$. We start by showing that the intersection $x\P_I w \, {}^F \P_I \, {}^Fx^{-1} \cap \U_J v'$ is empty if $v'$ and $v$ are not in the same $W_{K_x}$-coset of $W_J$. Since $wF$ normalises $I$, the element $xw\, {}^F x^{-1}F$ normalises $W_{K_x}$ and so does $vF$. Thus we can write

\centers{$ x\P_I w \, {}^F \P_I \, {}^F x^{-1} \cap \U_J v' \, = \,  x\B w\, {}^F W_I \B\, {}^F x^{-1} \cap \U_J v' = ({}^x\B)  v \, {}^F x \, {}^F W_I \B\, {}^F x^{-1} \cap \U_J v'.$}

\noindent By multiplying by ${}^F x \B$, we observe that if this set is non-empty, then one of the following double Bruhat cells

\centers{$ ({}^x\B)  v \, {}^F x \, {}^F W_I \B \cap \B v' \, {}^F x \B$}

\noindent is also non-empty. By Theorem \ref{deodec}, this means that there exists an $x^{-1}$-distin\-guished subexpression $\gamma$ of $v'\, {}^F x$ such that the product of the elements of $\gamma$ lies in the coset $v \, {}^F x \,{}^F W_I$. Since $x^{-1}$ is reduced-$J$, this subexpression has to start with a reduced decomposition of $v'$. The product of its elements is therefore of the form $v' \, {}^F x'$ with $x' \leq x$ for the Bruhat order. But then $v' \, {}^F x' \in v \, {}^F x \,{}^F W_I$ so that $x'$ is in the double coset $W_J x W_I$. This forces $x=x'$  since $x$ is the minimal element of this coset. Now, since $W_{K_x} = W_J \cap (W_I)^x$, the condition $v' \, {}^F x \in v\, {}^F x \, {}^F W_I$ implies $v' \in v \, {}^F W_{K_x}$ which, with $vF$-normalising $W_{K_x}$ is equivalent to $v' \in W_{K_x} v$.

\sk

Now, if we assume that $v'$ is $K_x$-reduced, we must have $v'=v$. In this case, the intersection $x\P_I w \, {}^F \P_I \, {}^Fx^{-1} \cap \U_J v$ is just $x \B x^{-1} v\, {}^F x \B \, {}^F x^{-1} \cap \U_J v$. The Curtis-Deodhar cell $\Omega_\gamma$ associated to the unique $x$-distinguished subexpression of $ x^{-1} v\, {}^F x$ giving $ x^{-1} v\, {}^F x$ is contained in $\U \cap \U^x$. Since the product ${}^x \Omega_\gamma \cdot (\U\cap {}^{v\, {}^Fx} \U)$ is stable by left multiplication by $\U \cap {}^x \U$, we deduce that 

\centers{$ x \B x^{-1} v\, {}^F x \B \, {}^F x^{-1} v^{-1} \cap \U = (\U \cap {}^x \U) \cdot (\U \cap {}^{v \,{}^F x} \U).$}

\noindent Finally, we can write $\U \cap {}^{v \,{}^F x} \U = (\U \cap \L_J \cap {}^{v \,{}^F x} \U )\cdot ( \U_J\cap {}^{v \,{}^F x} \U)$ and use the fact that $\U \cap \L_J \subset \U \cap {}^x \U$ to obtain

\centers{$ x\P_I w \, {}^F \P_I \, {}^F x^{-1} v^{-1} \cap \U_J = (\U_J \cap {}^x \U) \cdot (\U_J \cap {}^{v \,{}^F x} \U).$}

\noindent This proves that the fibers of $U_J \backslash \Z_x \, / \,(\L_J \cap \P_{K_x})  \twoheadrightarrow \X_{\L_J}(K_x,vF)$ are affine spaces of  same dimension.

\begin{rmk}\label{attentionlong}Note that the previous statement remains true if we replace ${}^F x$ by $x'$ with $\ell(x') = \ell(x)$. More precisely, if $I^w = I'$ and $xwx'^{-1} = v \in W_J$ is such that $(K_x)^v = K_{x'}$ then $x\P_I w \, \P_{I'} \, x'^{-1} v'^{-1} \cap \U_J$ is empty unless $v'$ and $v$ are in the same $W_{K_x}$-coset and in that case

\centers{$ x\P_I w \, \P_{I'} \, x'^{-1} v^{-1} \cap \U_J = (\U_J \cap {}^x \U) \cdot (\U_J \cap {}^{v x'} \U).$}

\noindent The condition $\ell(x) =  \ell(x')$ is essential, as several $W_{K_x}$-cosets of $W_J$ might be involved otherwise.

\end{rmk}

\sk

\noindent \thesubsection.2. \textbf{Coxeter elements for split groups.} Let $\{t_1,\ldots,t_n\}$ be the set of simple reflections associated to the basis $\Delta$ of the root system. Let $w=t_1 \cdots t_n$ be a  Coxeter element. We claim that all the pieces of $\X(w)$ but one are empty: by Remark \ref{rmktest} applied to $J = \emptyset$, the quotient $U \backslash \X_x$ is non-empty if and only if there exists an $x$-distinguished subexpression of $w$ whose product is trivial. But the only subexpression of $w$ whose product is trivial is $(1,1,\ldots,1)$, and it is $x$-distinguished for $x= w_0$ only.

\sk

Now let $J$ be a subset of $S$ and let $x = w_J w_0$ be the element of minimal length in $W_J w_0$. Let $v \in W_J$ be such that there exists an $x$-distinguished subexpression of $w$ whose product is $v^x \in (W_J)^{w_0}$. Denote by  $\widetilde J= \{t_{j_1},\ldots,t_{j_m}\}$ the conjugate of $J$ by $w_0$. Then $\gamma_i = t_i $ forces $t_i \in  \widetilde J$; furthermore, since $\gamma$ is $x$-distinguished then $\gamma_i = 1$ forces $t_i \notin \widetilde J$. We deduce that such a subexpression is unique and that $v= {}^x (t_{j_1} \cdots t_{j_m})$ is a Coxeter element of $W_J$.

\sk

For this subexpression, the cell $\Omega_\gamma$ is the ordered product of the groups $\U_i = u_{\gamma_1 \cdots \gamma_i(-\alpha_i)}(?)$ where $? = \F$ is  $\gamma_i \neq 1$ and $? = \F^\times$ otherwise. Note that when $i < j_b$ and $t_i \notin \widetilde J$, the groups $\U_{i}$ and $\U_{j_b}$ commute. Indeed, a positive combination of $\gamma_1 \cdots \gamma_{i}(-\alpha_{i}) = t_{j_1} \cdots t_{j_{a}}(-\alpha_{i})$ and $\gamma_1 \cdots \gamma_{j_b}(-\alpha_{j_b}) = t_{j_1} \cdots t_{j_{b-1}}(\alpha_b)$ is never a root since a positive combination of $-\alpha_i \in S\smallsetminus \widetilde J$ and $t_{j_{a+1}} \cdots t_{j_{b-1}}(\alpha_b) \in \Phi_{\widetilde J}^+$ never is. Furthermore, $\U\cap {}^{v x} \U = \L_J \cap \U \cap {}^v \U$ and it is not difficult to show that this group commutes with the groups ${}^x \U_i$ whenever $t_i \notin \widetilde J$. As a consequence

\centers{$ \Upsilon_\gamma \, = \, \varpi_J\big({}^x \Omega \cdot \U \cap {}^{ vx} \U\big) \, =\, \displaystyle \prod_{t_i \in S \smallsetminus \widetilde J} \, u_i(\F^\times).$}

\noindent We deduce that the morphism $U_J \backslash \Z_x = U_J \backslash \X(w) \longrightarrow \X_{\L_J}(v)$ has fibers isomorphic to $(\mathbb{G}_m)^{|S|-|J|}$. In \cite{Lu}, Lusztig actually constructs an isomorphism between $U_J \backslash \X(w)$ and  $\X_{\L_J}(v) \times (\mathbb{G}_m)^{|S|-|J|}$, but which is not compatible with the action of $L_J$. However, he proves that the cohomology groups of these two varieties are isomorphic as $L_J$-modules \cite[Corollary 2.10]{Lu}.

\mk

\noindent \thesubsection.3. \textbf{$n$-th roots of $\boldsymbol \pi$ for groups of type $A_n$.} Assume that $(\G,F)$ is a split group of type $A_n$. We denote by $t_1,\ldots,t_n$ the simple reflections of $W$ with the convention that there exists an isomorphism $W\simeq \mathfrak{S}_{n+1}$ sending the reflection $t_i$ to the transposition $(i,i+1)$. Let $J = \{t_1,\ldots,t_{n-1}\}$ and $w = t_1 t_2 \cdots t_{n-1} t_n t_{n-1}$ be a $n$-regular element. The $J$-reduced elements are of the form $x_i = t_n t_{n-1} \ldots t_i$ for $i=1,\ldots,n+1$. If $i\neq 1,n$, then $x_i < x_i t_1 < x_i t_1 t_2 < \cdots < x_i w$ and therefore the only $x_i$-distinguished subexpression of $w$ is $(t_1,t_2, \ldots, t_n,t_{n-1})$. Since ${}^{x_i} w \notin W_J$, we deduce from Remark \ref{rmktest} that the pieces $\X_{x_i}$ are empty. 

\sk

If $i=n$, then there are two $x_n$-distinguished subexpressions of $w$, namely $(t_1,t_2, \ldots, t_n,t_{n-1})$ and $(t_1,t_2, \ldots, t_n,1)$. But only one will give an element of $W_J$, since ${}^{x_n} (t_1 \cdots t_n) \notin W_J$ whereas ${}^{x_n} w = t_1 t_2 \cdots t_{n-1}$. By the example  \hyperref[sec24]{\ref*{sec24}.1}, the cohomology of $U\backslash \X_{x_n}$ is then, up to shift, isomorphic to the cohomology of the Coxeter variety $\X_{\L_J}(t_1 \cdots t_{n-1})$.

\sk

If $i = 1$ then $x_1=w_J w_0$. In that case there are many distinguished subexpressions of $w$. However, only one has a product in $(W_J)^x = W_{\{t_2, \ldots,t_n\}}$. Indeed, that condition forces $\gamma_1$ to be $1$ and therefore $\gamma=(1,t_2,\ldots,t_n,t_{n-1})$ is the only $x_1$-distinguished subexpression of $w$ whose product lies in $(W_J)^x$. For that subexpression, the Curtis-Deodhar cell ${}^x(\Omega_\gamma)$ is the product of $u_{\alpha_1 + \cdots + \alpha_n}(\mathbb{G}_m)$ with some affine subspace of $\L_J \cap \U$. Since $\alpha_1 + \cdots + \alpha_n$ is the longest root, the group $\L_J \cap \U$ acts trivially on $\U_{\alpha_1 + \cdots + \alpha_n}$ and we obtain $\Upsilon_\gamma = u_{\alpha_1 + \cdots + \alpha_n}(\mathbb{G}_m) \simeq \mathbb{G}_m$.

\sk

As in the Coxeter case, the varieties $U_J \backslash \X_{x_1}$ and $\X_{\L_J}( t_1 t_2 \cdots t_{n-2}t_{n-1} t_{n-2}) \times \mathbb{G}_m $ can be shown to have the same cohomology (see \cite[Proposition 8.17]{DM2}) but are non-isomorphic as $L_J$-varieties. However, there is a good evidence that such an isomorphism should hold for some Galois coverings of $\X$ and $\mathbb{G}_m$. We shall make this statement precise in the next section (see Section \ref{sec33} for an application to this example).

\section{\label{sec3}Lifting the decomposition to \texorpdfstring{$\widetilde \X(I,\dot wF)$}{Y(I,wF)}}

Recall that  one can associate to $\mathbf{I} \mathop{\longrightarrow}\limits^\mathbf{b} {}^F \mathbf{I}$ a variety $\widetilde \X(\mathbf{I},\mathbf{b} F)$ together with a Galois covering $\pi_{\mathbf{b}} \, : \, \widetilde \X(\mathbf{I},\mathbf{b} F) \longrightarrow \X(\mathbf{I},\mathbf{b} F)$ with Galois group $\L_I^{t(\mathbf{b}) F}$. Using this map one can pullback the previous constructions. More precisely, one can define the varieties $\widetilde \X_\mathbf{x} = \pi_\mathbf{b}^{-1} (\X_\mathbf{x})$ in order to obtain a partition of $\widetilde \X(\mathbf{I},\mathbf{b} F)$ into locally closed $P_J \times \L_I^{t(\mathbf{b})F}$-subvarieties. Furthermore, we can lift  the definition of $\Z_\mathbf{x}$ by considering the following cartesian diagram:

\begin{equation}\label{diagxz}\begin{psmatrix}  \widetilde \Z_\mathbf{x} & \Z_\mathbf{x} \\
					\widetilde \X_\mathbf{x} & \X_\mathbf{x}  
\psset{arrows=->>,nodesep=3pt} 
\everypsbox{\scriptstyle} 
\ncline{1,1}{1,2}^{/ \, \L_I^{t(\mathbf{b})F}}
\ncline{2,1}{2,2}^{/ \, \L_I^{t(\mathbf{b})F}}
\ncline{1,1}{2,1}
\ncline{1,2}{2,2}		
\end{psmatrix}
\end{equation}
\noindent For example, when $\mathbf{b} = \mathbf{w} \in \mathbf{W}$,  we can identify $\P_I / \U_I$ with $\L_I$ to construct $\widetilde \Z_x$ explicitly by

\centers{$ \widetilde \Z_x = \big\{ (p,m) \in \P_J \times {}^x \L_I \, \big| \, (pm)^{-1}\, {}^F (pm) \in \dot x \big(\U_I \dot w \, {}^F \U_I\big) \, {}^F \dot x^{-1}\big\}.$}

\noindent where the action of $\L_J \cap {}^x \P_I$ is given by $(p,m) \cdot l = (pl,l^{-1}m)$ with the convention that $\L_J \cap {}^x \U_I$ acts trivially on $m$. With this description, the map $\widetilde \Z_x \longrightarrow \widetilde \X_x$ is then given by $(p,m) \longmapsto pm \dot x\U_I$. Unlike the case of $\X_\mathbf{x}$, it is unclear whether there always exists a precise relation between quotients of $\widetilde \X_\mathbf{x}$ and smaller parabolic Deligne-Lusztig varieties. We shall therefore restrict ourselves to the following particular cases:

\begin{description}

\item[\textbf{Case 1.}] If $v= xw {}^F x^{-1}$ lies in the parabolic subgroup $W_J$ then, as in the example \hyperref[sec24]{\ref*{sec24}.1}, the cohomology of  $U_J \backslash \widetilde \X_x$ is related to the cohomology of $\widetilde \X_{L_J}(K_x,\dot v F)$. In this situation $\L_{K_x}^{\dot v F} \simeq (\L_I \cap \L_J^x)^{\dot w F}$ is a split Levi subgroup of   $\L_I^{\dot w F}$ so that one can modify $\widetilde \X_{L_J}(K_x,\dot v F)$ in order to obtain an action of $\L_I^{\dot w F}$  by Harish-Chandra restriction.

\item[\textbf{Case 2.}] If $w = sw'$ and $v= xw' {}^F x^{-1}$ lies in $W_J$, one can relate the varieties $U_J \backslash \widetilde \X_{x}$ and $\widetilde \X_{L_J}(K_x,\dot v F)$ (under some extra conditions on $s$ and $x$). The presence of $s$ is reflected by a Galois covering  of $\G_m$ which explains the geometry of the fiber in the examples \hyperref[sec24]{\ref*{sec24}.2} and \hyperref[sec24]{\ref*{sec24}.3}.  This covering carries actions of $\L_I^{w F}$ and $\L_I^{w' F}$ giving rise to a natural isomorphism $\L_I^{wF}/N \simeq \L_I^{\dot w' F}/N'$ as in \cite{BR1} in the case of tori.

\end{description}

 It turns out that this two rather specific cases are sufficient to study a large number of interesting Deligne-Lusztig varieties, namely the ones that are associated in \cite{BMi2} and \cite{DM3} to principal $\Phi_d$-blocks when $2d$ is strictly bigger than the Coxeter number. We shall give some examples in the Appendix for exceptional groups. The case of classical groups will be treated in a subsequent paper.

\subsection{Case 1 - Fibers are affine spaces}

We start under the assumptions of the example \hyperref[sec24]{\ref*{sec24}.1}. We assume that $x$ and $w$ satisfy $x w \, {}^F x ^{-1} \in W_J$. For simplicity, we shall also assume that this element is $W_{K_x}$-reduced, as it will always be the case in the examples.

\begin{prop}\label{ypartcase2}Assume that $v= x w \, {}^F x ^{-1}$ is a $W_{K_x}$-reduced element of $W_J$. Let $e = \dim (\U_J^x \cap {}^w \U \cap \U^-)$. Then there exists a group isomorphism $\L_I^{\dot w F} \simeq ({}^x \L_I)^{\dot v F}$ such that we have the following isomorphism in $D^b(\Lambda L_J \times( \L_{I}^{\dot wF} \rtimes \langle F \rangle)$-$\mathrm{mod})$:

\centers{$\Rgc\big(U_J\backslash \widetilde \X_{x},\Lambda\big)[2e](-e) \, \simeq \, \Rgc\big(\widetilde \X_{\L_J}(K_x, \dot v F), \Lambda\big)\  {\ol}_{\Lambda \P_J \cap ({}^x \L_{I})^{\dot v F}}\ \Lambda \L_{I}^{\dot w F}.$}
\end{prop}

\begin{proof} Since $v= x w \, {}^F x ^{-1}$, one can use Lang's Theorem to find an element $n \in N_\G(\T)$ such that $\dot v = n \dot w \, {}^F n^{-1}$. Then the conjugation by $n$ induces an isomorphism $\L_I^{\dot w F} \simeq ({}^x\L_I)^{\dot v F}$. Moreover, the map $(p,m) \in \widetilde \Z_x \, \longmapsto (p,m\dot x n^{-1})$ induces an isomorphism

\centers{$ \widetilde \Z_x  \simeq \big\{ (p,m) \in \P_J \times {}^x \L_I \, \big| \, (pm)^{-1}\, {}^F (pm) \in n \big(\U_I \dot w \, {}^F \U_I\big) \, {}^F n^{-1}\big\}$}

\noindent so that we can work with $n$ instead of $\dot x$. We shall relate the cohomology of this variety to the cohomology of $\widetilde \X_{\L_J}(K,\dot v F)$. For that purpose, we shall construct a morphism $\Psi \, : \, \widetilde \Z_x \longrightarrow \widetilde \X_{\L_J}(K,\dot v F) \times_{ \P_J \cap ({}^x \L_{I})^{\dot v F}}\  \L_{I}^{\dot w F}$ which will factor through $\widetilde \Z_x \longrightarrow U_J \backslash \widetilde \Z_x / \L_J \cap {}^x \P_I$ and then study its fibers. 

\sk

Let $(p,m) \in \widetilde \Z_x$. Since $p^{-1} \, {}^F p$ lies in $x \P_I w \, {}^F \P_I \, {}^F x^{-1}$ one can proceed as in the example \hyperref[sec24]{\ref*{sec24}.1} to show that it  also lies in the double coset $\P_{K_x} v \, {}^F \P_{K_x}$. If we write $p=ul \in \U_J \L_J$, we deduce that  $l^{-1} \, {}^F l \in (\L_J \cap {}^x \P_I)\, v \, {}^F( \L_J \cap {}^x \P_I)$. Therefore, there exists $l' \in \L_K = \L_J \cap {}^x {\L_I}$, unique up to multiplication on the right by $\L_K^{\dot vF}$ such that $(ll')^{-1} \, {}^F(ll') \in (\L_J \cap {}^x \U_I) \, \dot v \, {}^F (\L_J \cap {}^x \U_I)$. As a consequence, any element of $\widetilde \Z_x/\L_J \cap {}^x \P_I$ can be written $[p;m]$ where $p=ul$ is such that $l$ yields an element of $\widetilde \X_{\L_I}(K_x,\dot vF)$. For such a representative, we have

\centers{$p^{-1} \, {}^Fp \, = \, {}^{l^{-1}}(u^{-1} \, {}^Fu) \, (l^{-1} \, {}^F l) \, \in \,  (\L_J \cap {}^x \U_I)\cdot \U_J \, \dot v \, {}^F( \L_J \cap {}^x \U_I).$}

\noindent We can actually be more precise on the contribution of $\U_J$ in this decomposition. Indeed, we have seen in the example  \hyperref[sec24]{\ref*{sec24}.1} that  $x\P_I w \, {}^F \P_I \, {}^F x^{-1} v^{-1} \cap \U_J = (\U_J \cap {}^x \U) \cdot (\U_J \cap {}^{v \,{}^F x} \U)$ and hence

\centers{$p^{-1} \, {}^Fp \,  \in \,  (\L_J \cap {}^x \U_I)\cdot (\U_J \cap {}^x \U) \cdot (\U_J \cap {}^{v \,{}^F x} \U)\,  \dot v \, {}^F( \L_J \cap {}^x \U_I).$}

\noindent Now, the condition $(p,m) \in \widetilde \Z_x$ can be written $ p^{-1} \, {}^Fp \, \in \, m \, {}^{\dot v F} m^{-1} ({}^x \U_I)\, \dot v \, {}^F ({}^x \U_I)$ and we deduce that 

\centers{$  m \, {}^{\dot v F} m^{-1} \, \in \, {}^x \U_I \cdot (\U_J \cap {}^x \U) \cdot (\U_J \cap {}^{v \,{}^F x} \U) \cdot {}^{v\,{}^F x F} \U_I . $}

\noindent We want to show that $m \, {}^{\dot v F} m^{-1} \in \P_J$. For that purpose, we can decompose the intersection $\U_J \cap {}^{v \,{}^F x} \U$ into $\big( \U_J \cap {}^{v \,{}^F x F} (\L_I \cap \U) \big) \cdot (\U_J \cap {}^{v \,{}^F x F} \U_I)$ and we observe that $ \U_J \cap {}^{v \,{}^F x F} (\L_I \cap \U)  \subset {}^x \U$. Indeed, $x^{-1} vF(x) = w$ and by assumption $wF$ stablizes $\L_I \cap \U$. We deduce that 

\centers{$  m \, {}^{\dot v F} m^{-1} \, \in \, {}^x \U_I \cdot (\U_J \cap {}^x \U) \cdot {}^{v\,{}^F x F} \U_I. $}

\noindent Note that ${}^x \U_I \cdot (\U_J \cap {}^x \U)$ is contained in ${}^x \P_I$. In particular, the contribution of ${}^{v\,{}^F x F} \U_I$ in the decomposition of $m \, {}^{\dot v F} m^{-1}$ should also lie in ${}^x \P_I$. Since $wF$ normalises $\L_I$, the intersection $ {}^{v\,{}^F x F} \U_I \cap {}^x \P_I$ is contained in  ${}^x \U_I$. Finally, since $\L_I$ normalises $\U_I$ we deduce that $ m \, {}^{\dot v F} m^{-1} \, \in \, \U_J \cap {}^x \L_I$.
\sk

Therefore there exists $u' \in \U_J \cap {}^x{\L_I}$, unique up to multiplication by $\U_J \cap ({}^x\L_I)^{\dot vF}$ on the right, such that $u'^{-1}m \in ({}^x \L_I)^{\dot v F}$. To summarize, we have shown that to any pair $(p,m) \in \widetilde \Z_x$  one can associate a pair $(p',m')$ such that

\begin{itemize}

\item $(p,m)$ and $(p',m')$ are in the same  $\P_J \cap {}^x \L_I$-orbit, that is there exists $q \in \P_J \cap {}^x \L_I$ such that $p'=pq$ and $m' = q^{-1} p$;

\item the image of $p'$ by the composition $\P_J \longrightarrow \P_J / \U_J \simeq \L_J \longrightarrow \L_J / (\L_J \cap {}^x \U_I) $ lies in $\widetilde \X_{\L_J}(K,\dot v F)$;

\item  $m' \in {}^x \L_I$ is invariant by $\dot v F$.  

\end{itemize}

\noindent Moreover, if $(p'',m'')$ is any other pair satisfying the same conditions, then there exists $q' \in \P_J \cap ({}^x \L_I)^{\dot vF}$ such that $(p'',m'') = (p'q', q'^{-1}m')$ which means that $(p',m')$ is well defined in $\P_J \times_{\P_J \cap ({}^x \L_I)^{\dot v F}} ({}^x \L_J)^{\dot v F}$. Let us define now the morphism $\Psi$ by 

\centers{$ \Psi \, : \, (p,m) \in \widetilde \Z_x \, \longmapsto \, \big[\pi_J(p')\, (\L_J\cap {}^x \U_I)\, ; m'\big] \, \in \widetilde \X(K,\dot v F) \times_{\P_J \cap ({}^x \L_I)^{\dot v F}} ({}^x \L_J)^{\dot v F}$}

\noindent where the action of $\P_J \cap ({}^x \L_I)^{\dot v F}$ on $\widetilde \X_{\L_J}(K_x,\dot v F)$ is just the inflation of the action of $\L_{K_x}^{\dot vF} = \L_J \cap ({}^x \L_I)^{\dot v F}$. It is clearly surjective and equivariant for the actions of $P_J$ on the left and $({}^x \L_I)^{\dot v F}$ on the right. Furthermore, if $(p_1,m_1)$ and $(p_2,m_2)$ are in the same orbit under $\L_J \cap {}^x \P_I$, then $(p_1',m_1')$ and $(p_2',m_2')$ are in the same orbit under $\P_J \cap {}^x \P_I$. Let $q \in \P_J \cap {}^x \P_I$ be such that $(p_2',m_2') = (p_1' q , q^{-1} m_1')$  and write $q= u  l \in (\P_J \cap {}^x \U_I) \cdot (\P_J \cap {}^x \L_I)$. Then $l = m_1' {m_2'}^{-1} \in ({}^x \L_I)^{\dot vF}$ so that  $\Psi(p_1,m_1) = \Psi(p_2,m_2)$.  In other words, $\Psi$ induces a morphism

\centers{$ \widetilde \Z_x \, / \, \L_J \cap {}^x \P_I \, \longrightarrow \, \widetilde \X_{\L_J}(K,\dot v F) \times_{\P_J \cap ({}^x \L_I)^{\dot v F}} ({}^x \L_J)^{\dot v F}$}

\noindent which, in turn, yields a surjective equivariant morphism  

\centers{$U_J \backslash \widetilde \Z_x \, / \, \L_J \cap {}^x \P_I \, \longrightarrow \, \widetilde \X_{\L_J}(K,\dot v F) \times_{\P_J \cap ({}^x \L_I)^{\dot v F}} ({}^x \L_J)^{\dot v F}.$}

To conclude, it remains to study the fibers of this morphism. Since  $({}^x \L_J)^{\dot v F}$  acts freely on both varieties, we can rather look at the fibers of the map induced on the quotient varieties. Using the diagram  \ref{diagxz}, we can check that the latter  coincides with the map $\mathfrak{Z}_x^v = U_J \backslash \Z_x \, /\, \L_J \cap {}^x \P_I  \longrightarrow \X_{\L_J}(K_x, vF)$ which has affine fibers of dimension $r + \dim \U_J \cap {}^x \U$ (see Example  \hyperref[sec24]{\ref*{sec24}.1}).  
\end{proof}

\subsection{Case 2 - Minimal degenerations}

In this section we address the problem of computing the cohomology of the piece $\widetilde \X_x$ of $\widetilde \X(I,\dot w F)$ when $x w \, {}^F x^{-1}$ is close to be an element of $W_J$. Namely, we shall consider the following situation: $w = sw' > w' $ where $s \in S$ and $v = x w \, {}^F x^{-1} \in W_J$. Under some assumption on $s$ and $w'$ we will prove that the cohomology of $U_J \backslash \X_x$ and $\mathbb{G}_m \times \X_{\L_J} (K_x, vF)$ coincide. As we have seen in the examples, these two varieties are non-isomorphic in general. However, at the level of the varieties $\widetilde \X$ we shall construct a Galois covering $\widetilde{\mathbb{G}}_m \longrightarrow \mathbb{G}_m$ and a quasi-vector bundle

\centers{$ U_J \backslash \widetilde \X_x \, \rightsquigarrow \, \widetilde \X(K_x,\dot v F) \times_{\P_J \cap ({}^x \L_I)^{\dot v F}} \widetilde{\mathbb{G}}_m$}

\noindent such that $ \widetilde{\mathbb{G}}_m / \L_I^{\dot w F} \simeq \mathbb{G}_m$. As a byproduct, we will relate the cohomology of $U_J \backslash \widetilde \X_x$ and  $\mathbb{G}_m \times \X(K_x,\dot v F) $ with coefficients in any unipotent local system.

\sk

Throughout this section, we will assume that $[\G,\G]$ is simply connected. This is not a strong assumption since it has no effect on the unipotent part of the cohomology of a Deligne-Lusztig variety (see for example \cite[Section 5.3]{BR1}).  

\mk

\noindent \thesubsection.1. \textbf{Galois coverings of tori.} Let $\mathbf{I} \mathop{\longrightarrow}\limits^\mathbf{b} {}^F \mathbf{I}$, decomposed as  $\mathbf{I} = \mathbf{I}_1 \mathop{\longrightarrow}\limits^{\mathbf{w}_1}  \mathbf{I}_2  \mathop{\longrightarrow}\limits^{\mathbf{w}_2} \cdots $ $ \mathop{\longrightarrow}\limits^{\mathbf{w}_r} \mathbf{I}_{r+1}= {}^F \mathbf{I}$. Let  us consider an element $\mathbf{c} \in B^+$ obtained by minimal degenerations of the $w_i's$: we assume that $\mathbf{c} = \mathbf{z}_1 \cdots \mathbf{z}_r$ where $z_i = \gamma_i w_i$ with $\gamma_i \in S \cup \{1\}$ and $\ell(\gamma_i w_i) \leq \ell(w)$. We will also assume that each $\gamma_i$ commutes with $I_i$ so that $\mathbf{cF}$ normalises $\mathbf{I}$. Following \cite[Section 4]{BR1}, we set $\alpha_{\mathbf{b},\mathbf{c},i} = \alpha$ if $\gamma_i = s_\alpha$ or $\alpha_{\mathbf{b},\mathbf{c},i} = 0$ if $\gamma_i = 1$ and we define the following algebraic variety

\centers{$\mathbf{S}_{\mathbf{I} , \mathbf{b},\mathbf{c}} \, = \, \left\{ (l_1,\ldots,l_r) \in \L_{I_1} \times \cdots \times \L_{I_r} \, \left| \begin{array}{l}  l_i^{-1} \big({}^{\dot w_i}l_{i+1}\big) \in \mathrm{Im}\, \alpha_{\mathbf{b},\mathbf{c},i}^\vee \ \ \text{if } 1\leq i \leq r-1 \\[4pt] 
l_r^{-1} \big({}^{\dot w_r F} l_1\big) \in  \mathrm{Im}\, \alpha_{\mathbf{b},\mathbf{c},r}^\vee \end{array} \right. \right\}\cdot$}

\noindent Note that the assumption on  $\gamma_i$ ensures that the torus $\mathrm{Im}\, \alpha_{\mathbf{b},\mathbf{c},i}^\vee$ is central in $\L_{I_i}$, and therefore $\mathbf{S}_{\mathbf{I} , \mathbf{b},\mathbf{c}}$ is an algebraic group. 

\sk

Recall that $\L_I^{t(c)F}$ can be identified with ${\L'}^{\mathbf{c}F'}$ where $\L = \L_{I_1} \times \cdots \times \L_{I_r}$ and $\mathbf{c}F' : (l_1,\cdots ,l_r) \longmapsto \big({}^{\dot z_1} l_2, \ldots, {}^{\dot z_{r-1}} l_r, {}^{\dot z_r F} l_1\big)$.  The condition $ l_i^{-1} \big({}^{\dot w_i}l_{i+1}\big) \in \mathrm{Im}\, \alpha_{\mathbf{b},\mathbf{c},i}^\vee$ is equivalent to $ l_i^{-1} \big({}^{\dot z_i}l_{i+1}\big) \in \mathrm{Im}\, \alpha_{\mathbf{b},\mathbf{c},i}^\vee$ so that we can replace $w_i$ by $z_i$ in the definition of $\mathbf{S}_{\mathbf{I} , \mathbf{b},\mathbf{c}}$. In particular, the variety  $\mathbf{S}_{\mathbf{I} , \mathbf{b},\mathbf{c}}$ defines two Galois coverings of the torus $\prod \mathrm{Im}\, \alpha_{\mathbf{b},\mathbf{c},i}^\vee$, namely $\pi_\mathbf{b} : l \longmapsto l^{-1}\, {}^{\mathbf{b} F'} l$ and $\pi^\mathbf{c} : l \longmapsto \big({}^{\mathbf{c} F'} l\big) l^{-1}$, with respective Galois groups $\L_I^{t(\mathbf{b})F}$ and $\L_I^{t(\mathbf{c})F}$. We will denote by $d = \ell(\mathbf{b})-\ell(\mathbf{c})$ the dimension of this torus. 
Note that the induced action of $\L_I^{t(\mathbf{b})F}$ and $\L_I^{t(\mathbf{c})F}$ on $\mathbf{S}_{\mathbf{I},\mathbf{b},\mathbf{c}}$ is explicitely given by

\centers{$ (m,m')\cdot (l_1,\ldots,l_r) = \big(ml_1 m'^{-1},  (m^{\dot w_1})\,  l_2 \, (m'^{-1})^{\dot z_1} ,\ldots, (m^{\dot w_1  \cdots \dot  w_{r-1}}) \, l_r \, (m'^{-1})^{\dot z_1 \cdots \dot  z_{r-1}} \big)$}

\noindent for $m \in \L_I^{t(\mathbf{b})F}$ and $m' \in \L_I^{t(\mathbf{c})F}$. 

\sk

Let $\mathbf{S}_{\mathbf{I},\mathbf{b},\mathbf{c}}^\circ$ be the identity component of $\mathbf{S}_{\mathbf{I},\mathbf{b},\mathbf{c}}$. Since $\mathbf{S}_{\emptyset, \mathbf{b},\mathbf{c}} = \T^r \cap \mathbf{S}_{\mathbf{I},\mathbf{b},\mathbf{c}}$ is an $d$-dimensional closed subvariety of $\mathbf{S}_{\mathbf{I},\mathbf{b},\mathbf{c}}$ (it is also a Galois covering of $\prod \mathrm{Im}\, \alpha_{\mathbf{b},\mathbf{c},i}^\vee$)  it must contain the identity component $\mathbf{S}_{\mathbf{I},\mathbf{b},\mathbf{c}}^\circ$. This forces the stabilizer $N$ (resp. $N'$) of $\mathbf{S}_{\mathbf{I},\mathbf{b},\mathbf{c}}^\circ$ in $\L_I^{t(\mathbf{b})F}$ (resp. $\L_I^{t(\mathbf{c})F}$) to be contained in $\T$. In particular, we can readily extend the results in \cite[Section 4.4.3]{BR1} to obtain an explicit description of $N$ and $N'$ in terms of sublattices of $Y(\T)$. For example one can check that $W_I$ acts trivially on these lattices so that $N$ are $N'$ are normal subgroups of $\L_I$.

\sk

It turns out that the covering $\mathbf{S}_{\mathbf{I}, \mathbf{b},\mathbf{c}}$ will naturally appear in the quotient of the parabolic Deligne-Lusztig varieties that we will consider. The action of $\L_I^{t(\mathbf{b})F}$ and $\L_I^{t(\mathbf{c})F}$ yields canonical isomorphisms $\L_I^{t(\mathbf{b})F}/ N \simeq \L_I^{t(\mathbf{c})F}/ N' \simeq \mathbf{S}_{\mathbf{I}, \mathbf{b},\mathbf{c}} / \mathbf{S}_{\mathbf{I}, \mathbf{b},\mathbf{c}}^\circ$. Let us write $\mathbf{S}_{\mathbf{I}, \mathbf{b},\mathbf{c}} = \L_I^{t(\mathbf{b}) F} \times_{N} \, \mathbf{S}_{\mathbf{I}, \mathbf{b},\mathbf{c}}^\circ$. The quotient of this variety by the action of $N$ (by left multiplication) is given by

\centers{$ N \backslash \mathbf{S}_{\mathbf{I}, \mathbf{b},\mathbf{c}} \, \simeq \, \L_I^{t(\mathbf{b}) F} / {N} \times \big( \prod \mathrm{Im}\, \alpha_{\mathbf{b},\mathbf{c},i}^\vee\big) \, \simeq \, \mathbf{S}_{\mathbf{I}, \mathbf{b},\mathbf{c}} / \mathbf{S}_{\mathbf{I}, \mathbf{b},\mathbf{c}}^\circ \times \big( \prod \mathrm{Im}\, \alpha_{\mathbf{b},\mathbf{c},i}^\vee\big).$}

\noindent On this quotient, $\L_I^{t(\mathbf{b}) F} /{N}$  acts on the first factor only but the action of $\L_I^{t(\mathbf{c}) F}$ is more complicated: an element $m \in \L_I^{t(\mathbf{c}) F}$ acts on  $ \prod \mathrm{Im}\, \alpha_{\mathbf{b},\mathbf{c},i}^\vee$ by mulitplication by $(m ({}^{\gamma_1} m^{-1}), (m^{z_1})\, {\vphantom{\big(}}^{\gamma_2} \big((m^{-1})^{z_1}\big), \ldots, (m^{z_1 \cdots z_{r-1}})\, {\vphantom{\big(}}^{\gamma_r} \big((m^{-1})^{z_1 \cdots z_{r-1}}\big) \big)$.  This action can be extended to the connected group $\L_I$. Consequently, if the order of $\L_I^{t(\mathbf{c})F}$ is invertible in $\Lambda$, then the cohomology of $ N \backslash \mathbf{S}_{\mathbf{I}, \mathbf{b},\mathbf{c}} $ can be represented by a complex with a trivial action of $N'$ and we have
\begin{equation}\label{cover} \Rgc( N \backslash \mathbf{S}_{\mathbf{I}, \mathbf{b},\mathbf{c}} , \Lambda) \,  \simeq \,  \Rgc( N \backslash \mathbf{S}_{\mathbf{I}, \mathbf{b},\mathbf{c}} /N',\Lambda) \, \simeq \, \Lambda \mathbf{S}_{\mathbf{I}, \mathbf{b},\mathbf{c}} / \mathbf{S}_{\mathbf{I}, \mathbf{b},\mathbf{c}}^\circ \, {\ol}_\Lambda\, \Rgc\big((\mathbb{G}_m)^d,\Lambda\big) \end{equation}
\noindent in $D^b(\Lambda \L_I^{t(\mathbf{b})F}/N \times \L_I^{t(\mathbf{c})F}/N'$-$\mathrm{mod})$.

\bk

\noindent \thesubsection.2. \textbf{The model $w = sw'$.}   We start with the case $r=1$, that is when $\mathbf{b} = \mathbf{w} \in \mathbf{W}$. Let $x$ be a $J$-reduced-$I$ element of $W$ and $s \in S$ be such that $w' = sw < w$ and $v=xsw\, {}^F x^{-1} \in W_J$. Recall from the previous section that if $s$ acts trivially on $\Phi_I$, then there exists normal subgroups $N$ of $\L_I^{\dot wF}$  and $N' $ of $\L_I^{\dot w' F}$ together with a canonical isomorphism $\L_I^{\dot wF} / N \simeq \L_I^{\dot w'F} / N'$. Using these small finite groups one can relate the cohomology of $U_J \backslash \widetilde \X_x $ to the cohomology of $\widetilde \X_{\L_J}(K_x,\dot vF)$:

\begin{prop}\label{ypartcase1}Let $w$ be an $I$-reduced element of $W$ such that ${}^{w F} I = I$. Assume that $w$ can be decomposed into $w=sw'$ such that 

\begin{itemize}

\item[$\mathrm{(i)}$] $v = x w' \, {}^F x^{-1} \in W_J$ and $\ell(v) = \ell(w')$

\item[$\mathrm{(ii)}$] $s \in S $ acts trivially on $\Phi_I$

\item[$\mathrm{(iii)}$] ${}^x (W_I s ) \cap W_J = 1$

\end{itemize}

 \noindent Then there exists a group isomorphism $\L_I^{\dot w F}/N \simeq \L_I^{\dot w' F} /N'$ such that, if the order of $\L_I^{\dot w' F}$ is invertible in $\Lambda$, we have 

\centers{$  \Rgc\big(U_J \backslash \widetilde \X_{x} \, /   N , \Lambda\big)\, \simeq \, \Rgc\big(\mathbb{G}_m \times \widetilde \X_{\L_{J}}(K_x,\dot v F), \Lambda\big) \, {\ol}_{\Lambda (\P_J \cap {}^x \L_I)^{\dot v F}} \, \Lambda  \L_I^{\dot w' F}/ N'$}

\noindent in $D^b(\Lambda L_J \times (  \L_I^{\dot w F} / N \rtimes \langle F \rangle)$-$\mathrm{mod})$.

\end{prop}

\begin{proof} Let $v = x w'\, {}^F x^{-1} \in W_J$ and let $n$ be a representative of $x$ in $N_\G(\T)$ such that  $\dot v = n \dot w' {}^F x^{-1}$. As is the proof of Proposition \ref{ypartcase2}, we shall work with $n$ instead of $\dot x$ and identify the variety $\widetilde \Z_x$ with

\centers{$ \big\{ (p,m) \in \P_J \times {}^x \L_I \, \big| \, (pm)^{-1}\, {}^F (pm) \in n \big(\U_I \dot w \, {}^F \U_I\big) \, {}^F n^{-1}\big\}.$}

\noindent In order to compute the quotient by $U_J$, we need a precise condition on $u \in \U_J$, $l\in \L_J$  and $m$ for  $(ul,m)$ to belong to this variety. We start by proving the following:

\mk

\begin{lem} Under the assumptions of Proposition \ref{ypartcase1}, if  $(p,m)$ belongs to $\widetilde \Z_x$ then $m {}^{\dot v F} m^{-1}$ lies in $\P_J$. 
\end{lem}

\begin{proof}[Proof of the Lemma] Since $sw'$ is $I$-reduced, $s\notin I$ and $\U_{I} \dot s \subset \U_{\alpha_s} \dot s \U_{I}$. Therefore we can write 

\centers{$\U_{I} \dot s \dot  w' \, {}^F \U_I \, \subset \, \U_{\alpha_s} \dot s \U_I  \dot w' \, {}^F \U_I. $ }

\noindent Note that this inclusion is actually an equality: indeed, $w'^{-1}(\alpha_s) \in \Phi^+$ since $sw' > w'$ and $w'^{-1}(\alpha_s) \notin {}^F \Phi_I^+$ otherwise $-\alpha_s = sw' (w'^{-1}(\alpha_s))$ would be in $\Phi_I^+$ by assumption on $sw'$.

\sk

The double coset $ \U_I  \dot w' \, {}^F \U_I$ can also be simplified: for $a \in W$ we denote $N(a)=\{\alpha\in \Phi^+ \, | \, a^{-1}(\alpha) \in \Phi^-\}$. If $\ell(ab) = \ell(a) +\ell(b)$ then $N(ab) = N(a) \amalg aN(b)$. Using assumption  $\mathrm{(i)}$ we can apply this to $xw' = v\, {}^F x$ in order to obtain

\centers{$ xN(w') \, = \, N(xw') \smallsetminus N(x) \, = \, \big(N(v)\amalg vN({}^F x)\big) \smallsetminus N(x).$}

\noindent Since $v \in W_J$ and $x$ is $J$-reduced, the sets $N(v)$ and $N(x)$ are disjoint. Moreover, $N(x)$ and $N({}^F x)$ have the same number of elements and hence $xN(w') = N(v)$. This proves that $\U \cap {}^{w'}\U^- = (\U \cap {}^v \U^-)^x \subset \L_J^x$. Since $w'F$ (like $wF$ by assumption $\mathrm{(ii)}$) normalises $I$ we deduce that 
\begin{equation}\label{uwu} \U_{I} \dot  w' \, {}^F \U_I  \, = \, (\U_I \cap \L_J^x) \, \dot  w' \, {}^F \big((\U_I \cap \L_J^x) \cdot (\U_I \cap (\U_J^-)^x) \cdot (\U_I \cap \U_J^x)\big).
\end{equation}
\noindent Now let $p \in \P_J$ be an element of $m n \U_I \dot s \dot w' \,{}^F \U_I {}^F (mn)^{-1}$. There exists $l_s \in \U_{\alpha_s} \dot s$ such that  $p \in m n l_s  \, \U_I \, \dot w' \,{}^F \U_I  {}^F (mn)^{-1} $. Since $\L_I$ normalises $\U_I$, we have $p \in (m {}^n l_s {}^{\dot v F} m^{-1})   \, n \U_I \, \dot w' \,{}^F \U_I  {}^F n^{-1}$. Now, by \ref{uwu}, the class $n \U_I \, \dot w' \,{}^F \U_I  {}^F n^{-1}$ is contained in $\P_J^- \cdot \P_J$ and therefore $m {}^n l_s {}^{\dot v F} m^{-1} \in \P_J \cdot \P_J^-$. We claim that this forces $l_s \notin \T \dot s$.  Otherwise ${}^x(\L_I \, s \, \L_I) = {}^x (\L_I s)$ would have a non-trivial intersection with $\P_J \cdot \P_J^-$, which is impossible by the Bruhat decomposition since ${}^x(W_I s)$ and $W_J$ are disjoint. 

\sk

Let $\T_s$ be the image of $\alpha_s^\vee$. By a simple calculation in $\G_s = \langle \U_{\alpha_s}, \U_{-\alpha_s}\rangle$, we deduce that $l_s \in \U_{-\alpha_s} \T_s \U_{\alpha_s}$.  Since $s$ acts trivially on $\Phi_I$, the group the group $\L_I$ normalises $\U_{\alpha_s}$ and $\T_s = \mathrm{Im} \, \alpha_s^\vee$. Moreover, ${}^x \U_{\alpha_s} \subset \U_J^-$ and therefore  $m {}^{\dot v F} m^{-1} \in \P_J \cdot \P_J^-$. If we decompose $\L_I$ into $(\B_I, \B_I^-)$-orbits, we have, as $x$ is reduced-$I$
 
 \centers{$ {}^x(\L_I) \cap (\P_J \cdot \P_J^-) = \displaystyle \coprod_{v' \in W_{K_x}} {}^x \B_I v' {}^x \B_I^- = ({}^x \L_I \cap \P_J)\cdot ({}^x \L_I \cap \U_J^-).$}

\noindent We want to prove that the contribution  of $\U_J^-$ on $m {}^{\dot v F} m^{-1}$ is trivial. Write $m {}^{\dot v F} m^{-1} = m'm''$ with $m' \in {}^x \L_I \cap \P_J$ and $m'' \in {}^x \L_I \cap \U_J^-$. Using \ref{uwu} and the fact that $l_s \in \U_{\alpha_s} \T_s \U_{-\alpha_s}$, we see that there exists $l' \in ({}^x \U_I \cap \L_J)\, \dot v \, {}^F ({}^x \U_I \cap \L_J)$ such that $p \in {}^x (\U_{-\alpha_s} \T_s) \, m'm'' \, \U_{x(\alpha_s)} \, l' \, {}^F \big(({}^x\U_I \cap \U_J^-) \cdot ({}^x \U_I \cap \U_J)\big)$. In this decomposition, $ {}^x (\U_{-\alpha_s} \T_s)$, $m'$, $l'$ and ${}^F({}^x \U_I \cap \U_J)\big)$ lie in $\P_J$, whereas $m''$, $\U_{x(\alpha_s)}$ and ${}^{l'F}({}^x \U_I \cap \U_J)$ lie in $\U_J^-$. Since $\P_J \cap \U_J^-$ is trivial, we deduce  that $m'' \in {}^{l' F} \big({}^x \U_I \cap \U_J^- ) \cdot  \U_{x(\alpha_s)} $.  Finally, since ${}^x \U_I \cap \L_J$ normalises ${}^x \P_I \cap \U_J^-$ and both $m''$ and $\U_{x(\alpha_s)}$ are contained in this group, we can conclude if we can show that  $ ({}^x \P_I \cap \U_J^-) \cap \, {}^{vF}( {}^x \U_I \cap \U_J^-) \subset {}^x \U_I$.  But ${}^x \P_I \cap {}^{vF} ({}^x \U_I) = {}^x (\P_I \cap {}^{w'F} \U_I) = {}^x (\U_I \cap {}^{w' F} \U_I) $ since $w'F$ normalises $\L_I$. \end{proof}

\begin{lem} Under the assumptions of Proposition \ref{ypartcase1}, let $m \in {}^x \L_I$ and $l \in ({}^x \U_I \cap \L_J) \, \dot v \, {}^F ({}^x \U_I \cap \L_J)$. For $u \in \U_J$, the element  $ul$ lies in $mn\,\U_{I} \dot s \dot w \, {}^F \U_I \, {}^F(mn)^{-1}$  if and only if there exist $\lambda \in \F^\times$, $m_1 \in {}^x \L_I \cap \U_J$ and $u_1 \in {}^F ({}^x \U_I \cap \U_J)$  such that

\begin{itemize}
  \item $m \,{}^{\dot v F} m^{-1}  = m_1 \cdot {}^n \alpha_s^\vee (\lambda)$
  \item $u = {}^{mn} u_{-\alpha_s}(\lambda) \cdot m_1 \cdot {}^{l} u_1$.

\end{itemize}
\end{lem}

\begin{proof}[Proof of the Lemma] We have already seen in the course of the proof of the previous lemma (see \ref{uwu}) that $ul$ can be written $ul = ({}^{mn} l_s) \, (m \, {}^{\dot v F} m^{-1})  \, l' \, u_2 \, u_1$ with $l_s \in u_{\alpha_s}(\F^\times) \, \dot s$, $l' \in  ({}^x \U_I \cap \L_J) \, \dot v \, {}^F ({}^x \U_I \cap \L_J)$, $u_2 \in {}^F({}^x \U_I \cap \U_J^-) $ and $u_1 \in {}^F({}^x \U_I \cap \U_J)$. By a simple calculation in $\G_s= \langle \U_{\alpha_s},\U_{-\alpha_s}\rangle$ we can decompose $l_s$ into $l_s = u_{-\alpha_s}(\lambda) \, \alpha_s^\vee(\lambda^{-1}) \, u_{\alpha_s}( -\lambda)$ where $\lambda \in \F^\times$ is uniquely determined (note that we have chosen specific $u_{\alpha}$'s in Section \ref{sec1}). By the previous Lemma $m \, {}^{\dot vF}m^{-1} = m_1 m_2$ with $m_1 \in {}^x \L_I \cap \U_J$ and $m_2 \in {}^x\L_I \cap \L_J$. From the expression of $ul$ we obtain

\centers{$ m_1^{-1} \, ({}^{mn} u_{-\alpha_s}(-\lambda))\,  u\, {}^l u_1^{-1} \, =\, {}^{m_1^{-1}mn} \big(\alpha_s^\vee(\lambda^{-1})u_{\alpha_s}( -\lambda )\big) m_2 l' l^{-1} {}^l u_2$.}

\noindent Since $\L_J$ (resp. $\L_I$) normalises $\U_J$ (resp. $\U_{\alpha_s}$ and $\U_{-\alpha_s}$) and $\U_{x(-\alpha_s)} \subset \U_J$, the left hand-side of this equality lies in $\U_J$ whereas the right-hand side lies in $\P_J^-$. Therefore it must be trivial and we obtain 

\begin{itemize}

\item $u = {}^{mn} u_{-\alpha_s}(\lambda)\, m_1 \, {}^l u_1$;

\item ${}^{m_1^{-1}mn} \big(\alpha_s^\vee(\lambda^{-1})\big) m_2 l' l^{-1} = 1$ and therefore $l = l'$ and $m_2 = {}^{m_1^{-1}mn} \big(\alpha_s^\vee(\lambda)\big) = {}^n \alpha_s^\vee (\lambda)$; 

\item ${}^{m_1^{-1}mn\alpha_s^\vee(\lambda^{-1})} \big(u_{\alpha_s}(-\lambda )\big) \, {}^{m_2l} u_2 = 1 $ and hence $u_2 = {}^{l^{-1}m_1^{-1} m n }  \big(u_{\alpha_s}(-\lambda )\big)$.
\end{itemize}

\noindent Conversaly, one can readily check that if these relations are satisfied then $ul \in mn\,\U_{I} \dot s \dot w \, {}^F \U_I \, {}^F(mn)^{-1}$. \end{proof}
\sk

As a consequence of the lemmas, we can proceed as in the proof of Proposition \ref{ypartcase2} to show that any element of $\widetilde \Z_x$ is in the $\P_J \cap {}^x \L_I$-orbit of a element $(p,m) = (ul,m)$ satisfying the following properties: 

\begin{itemize}

\item $l^{-1} \, {}^F l \in \, ({}^x \U_I \cap \L_J) \, \dot v \, {}^F ({}^x \U_I \cap \L_J)$
  \item $m \,{}^{\dot v F} m^{-1}  =  {}^n \alpha_s^\vee (\lambda)$
  \item $\big(u^{-1} \, {}^F u\big)^l = \big({}^{mn} u_{-\alpha_s}(\lambda)\big)\, \big({}^{l^{-1} \, {}^F l} u_1\big)$

\end{itemize}

\noindent for some $\lambda \in \F^\times$ and $u_1 \in {}^F( {}^x \U_I \cap \U_J)$ both uniquely determined. Moreover, the elements of this form in the class of $(p,m)$ form a single $(\P_J \cap {}^x \L_I)^{\dot v F}$-orbit.  

\sk

Recall from the previous section that to $w$ and $w'$ one can associated an algebraic group $\mathbf{S}_{\mathbf{I},\mathbf{w},\mathbf{w'}}$ above $\mathbb{G}_m$ defined by $\mathbf{S}_{\mathbf{I},\mathbf{w},\mathbf{w'}} = \{m \in \L_I \, | \, m^{-1} \,  {}^{\dot w F} m  \in \T_s \}$. Using the special representatives of $ \widetilde \Z_x/\P_J \cap {}^x \L_I $ mentioned above, we can define the following map

\centers{$ \Psi \, : \, [p;m] \in \widetilde \Z_x/\P_J \cap {}^x \L_I \, \longmapsto \, \big[l\, (\L_J\cap {}^x \U_I)\, ; m^{-1}\big] \, \in \widetilde \X(K_x,\dot v F) \times_{\P_J \cap ({}^x \L_I)^{\dot v F}} {}^n \mathbf{S}_{\mathbf{I},\mathbf{w},\mathbf{w'}}$}

\noindent where the action of $\P_J \cap ({}^x \L_I)^{\dot v F}$ on $\widetilde \X_{\L_J}(K_x,\dot v F)$ is just the inflation of the action of $\L_{K_x}^{\dot vF} = \L_J \cap ({}^x \L_I)^{\dot v F}$. It is clearly surjective and equivariant for the actions of $P_J$ and $ {}^n (\L_I^{\dot w F})$. The quotient by $U_J$ (which acts trivially on $\widetilde \X_{\L_J}(K_x,\dot v F)$) gives rise to a surjective $L_J \times {}^n(\L_I^{\dot w F})$-equivariant morphism
\begin{equation} \label{morphaff} U_J \backslash \widetilde \Z_x/\P_J \cap {}^x \L_I \, \longrightarrow \, \widetilde \X(K_x,\dot v F) \times_{\P_J \cap ({}^x \L_I)^{\dot v F}} {}^n \mathbf{S}_{\mathbf{I},\mathbf{w},\mathbf{w'}}.
\end{equation}
\noindent Furthermore, any element $[U_J ull';m]$ in the fiber of $[ l\, (\L_J \cap {}^x \U_I) ;m^{-1}]$ is uni\-quely determined by an element $l' \in ({}^x \U_I \cap \L_J)$ and $u^{-1} {}^Fu$. Since the latter is determined by $u_1 \in {}^F ({}^x \U_I \cap \U_J)$, we deduce that the fibers are affine spaces of dimension $\dim ({}^x \U_I \cap \P_J)$. By comparing the dimensions,  we obtain the following isomorphism in $D^b(L_J $-$\mathrm{mod}$-${}^n(\L_I^{\dot w F}))$:

\centers{$\Rgc(U_J \backslash \widetilde \X_x, \Lambda) \, \simeq \, \Rgc\big(\widetilde \X(K_x,\dot v F) , \Lambda\big) \, {\ol}_{\Lambda \P_J \cap ({}^x \L_I)^{\dot v F}} \, \Rgc({}^n \mathbf{S}_{\mathbf{I},\mathbf{w},\mathbf{w'}}, \Lambda).$}

\noindent  and we conclude using \ref{cover}, which gives the cohomology of $N \backslash \mathbf{S}_{\mathbf{I},\mathbf{w},\mathbf{w'}}$ with the action of  $\L_I^{\dot w F}/ N$ and $\L_I^{\dot w' F} / N'$. \end{proof}

\begin{rmk} In many cases we will use this Lemma under the assumption that either $I = \emptyset$ or  $x = w_0 w_J$. This extra condition makes the previous proof much simpler.
\end{rmk}

\begin{rmk} When $[\G,\G]$ is not simply connected, the coroot $\alpha_s^\vee$ might not be injective. In that case, the fibers of the morphism \ref{morphaff} are not necessarily affine spaces. To obtain an analogous statement, we need to change slightly the definition of $\mathbf{S}_{\mathbf{I},\mathbf{w},\mathbf{w'}}$ and consider instead $\big\{ (m, \lambda) \in \L_I \times \G_m \, | \, m^{-1} \, {}^{\dot w F} m = \alpha_s^\vee (\lambda)\big\}$.
\end{rmk}

\sk

\noindent \thesubsection.3. \textbf{The main result.} More generally, one can combine Proposition \ref{ypartcase2} and \ref{ypartcase1} in order to obtain the following result for elements in the Braid monoid:

\begin{thm}\label{mainthm}Let $\mathbf{I} \mathop{\longrightarrow}\limits^\mathbf{b} {}^F \mathbf{I}$ decomposed as $\mathbf{I} = \mathbf{I}_1 \mathop{\longrightarrow}\limits^{\mathbf{w}_1}$ $ \mathbf{I}_2  \mathop{\longrightarrow}\limits^{\mathbf{w}_2} \cdots  \mathop{\longrightarrow}\limits^{\mathbf{w}_r} \mathbf{I}_{r+1}= {}^F \mathbf{I}$ and $\mathbf{c} = \mathbf{z}_1 \cdots \mathbf{z}_r$ obtained by minimal degenrations of the $w_i$'s. More precisely, we assume that $z_i = \gamma_i w_i$ with $z_i \leq w_i$ and $\gamma_i \in S\cup \{1\}$. Let $\mathbf{x}=(x_1,\ldots,x_r)$ be a $r$-tuple of $J$-reduced-$I_i$ elements of $W$ of same length and set $x_{r+1} = {}^F x_1$. We assume that

\begin{itemize}

\item if $\gamma_i = 1$ then $v_i = x_i w_i  x_{i+1}^{-1}$ is a $K_{x_i}$-reduced element of $W_J$;

\item if $\gamma_i \in S$ then the following properties are satisfied: 

\begin{itemize}

\item[$\mathrm{(i)}$] \vskip -3mm $v_{j} = x_{i} z_{i} x_{i+1}^{-1} \in W_J$ and $\ell(v_{i}) = \ell(z_{i})$

\item[$\mathrm{(ii)}$] $\gamma_i$ acts trivially on $\Phi_{I_{j}}$

\item[$\mathrm{(iii)}$] ${}^{x_{i}} (W_{I_i} s ) \cap W_J = 1$.

\end{itemize}
\end{itemize}

\noindent Let us denote by

\begin{itemize}
\item $e = \sum \dim (\U_J^{x_i} \cap {}^{z_i} \U \cap \U^-)$;

\item $d = \#\{i =1,\ldots,r \, | \, \gamma_i \in S\} = \dim \mathbf{S}_{\mathbf{I},\mathbf{b},\mathbf{c}}$;

\item $\mathbf{v} = \mathbf{v_1} \cdots \mathbf{v_r} \in B_{W_J}^+ $;

\item $N$ (resp. $N'$)  the stabiliser of $\mathbf{S}_{\mathbf{I},\mathbf{b},\mathbf{c}}^\circ$ in $\L_I^{t(\mathbf{b})F}$ (resp. $\L_I^{t(\mathbf{c})F}$).

\end{itemize}

\noindent  If the order of $\L_I^{t(\mathbf{c})F}$ is invertible in $\Lambda$, then there exists a natural isomorphism $\L_I^{t(\mathbf{b})F}/N \simeq \L_I^{t(\mathbf{c})F} / N'$ such that the cohomology of the piece $\widetilde \X_\mathbf{x}$ of the Deligne-Lusztig variety $\widetilde \X(\mathbf{I},\mathbf{b}F)$ satisfies

\centers{$  \Rgc\big(U_J \backslash \widetilde \X_{\mathbf{x}} \, /   N , \Lambda\big)[2e](-e)\, \simeq \, \Rgc\big((\mathbb{G}_m)^d \times \widetilde \X_{\L_{J}}(\mathbf{K}_{x_1},\mathbf{v} F)\big) \, {\ol}_{\Lambda (\P_J \cap {}^{x_1} \L_{I})^{t(\mathbf{v}) F}} \, \Lambda  \L_I^{t(\mathbf{c})F} / N'$}

\noindent in $D^b(\Lambda L_J \times (  \L_I^{t(\mathbf{b}) F} / N \rtimes \langle F \rangle)$-$\mathrm{mod})$.

\end{thm}

\begin{proof}[Sketch of proof] Recall that the piece $\widetilde \X_\mathbf{x}$ can be lifted up to a variety $\widetilde \Z_\mathbf{x}$ defined as the set of $2r$-tuples $(\mathbf{p},\mathbf{m}) = (p_1,\ldots,p_r,m_1,\ldots,m_r) \in (\P_J)^r \times {}^{x_1} \L_{I_1} \times \cdots \times {}^{x_r} \L_{I_r}$ such that 

\centers{$(p_im_i)^{-1}\, p_{i+1}m_{i+1}\  \in \ \dot x_{i} \big(\U_{I_i} \dot w_i \U_{i+1}\big)\dot x_{i+1}^{-1}$}

\leftcenters{and}{$(p_rm_r)^{-1}\, {}^F (p_1m_1)\  \in \ \dot x_r \big(\U_{I_r} \dot w_r \, {}^F \U_{I_1}\big) \, {}^F \dot x_1^{-1}.$}

\noindent As in the proofs of Proposition \ref{ypartcase2} and \ref{ypartcase1} (see also Remark \ref{attentionlong}), we can find good representatives in the $\prod \P_J \cap {}^{x_i}\L_{I_i}$-orbit of $(\mathbf{p},\mathbf{m})$, giving rise to a morphism

\centers{$ \Psi \, : \, \widetilde \Z_\mathbf{x} \, \longrightarrow \, \widetilde \X_{\L_{J}}(\mathbf{K}_{x_1},\mathbf{v} F) \times_{ (\P_J \cap {}^{x_1} \L_{I})^{t(\mathbf{v}) F}} \, \mathbf{S}_{\mathbf{I},\mathbf{b},\mathbf{c}}$}

\noindent which will factor via the quotient of $\widetilde \Z_\mathbf{x}$ by $U_J$ and $\prod \P_J \cap {}^{x_i}\L_{I_i}$ into a morphism whose fibers are affine spaces. Note that one can find $n_1 \in N_\G(\T)$ such that $n_1 t(\mathbf{c}) \, {}^F n_1^{-1} = t(\mathbf{v})$. The action of $({}^{x_1}\L_I)^{t(\mathbf{v})F}$ on $\mathbf{S}_{\mathbf{I},\mathbf{b},\mathbf{c}}$ is then given by the right action of $\L_I^{t(\mathbf{c})F} = \big(({}^{x_1}\L_I)^{t(\mathbf{v})F}\big)^{n_1}$ on $ \mathbf{S}_{\mathbf{I},\mathbf{b},\mathbf{c}}$. We conclude using  \ref{cover}.\end{proof}

\begin{rmk}\label{rmkunip} By definition, any unipotent character of $G$ appears in the cohomology of some Deligne-Lusztig variety. If $H$ is a normal subgroup of $\G$ contained in $T$, then $H$ acts trivially on $\G/\B$ and therefore any unipotent character of $G$ is trivial on $H$. This applies in particular to the subgroups $N$ and $N'$ of $\L_I^{t(\mathbf{b})F}$ and $\L_I^{t(\mathbf{c})F}$ so that they have the same unipotent characters. Now, the group $(\P_J)^{x_1} \cap  \L_I$ is a parabolic subgroup of $\L_I$, stable by $t(\mathbf{c}) F$, and it has  $\L_J^{x_1} \cap \L_I = (\L_{K_{x_1}})^{x_1}$ as a rational Levi complement. Therefore any unipotent character $\chi$ of $\L_I^{t(\mathbf{c})F}$ (or equivalently of $\L_I^{t(\mathbf{b})F}$)  has a Harish-Chandra restriction ${}^* \mathrm{R}_{K_{x_1}}^I \, \chi$ to $\L_{K_{x_1}}^{t(\mathbf{v})F}$ (after a suitable conjugation). With this notation, we obtain

\centers{$  \Rgc\big(\widetilde \X_{\mathbf{x}}, \qlb\big)_\chi^{U_J}\, \simeq \, \Rgc\big((\mathbb{G}_m)^d \times \widetilde \X_{\L_{J}}(\mathbf{K}_{x_1},\mathbf{v} F), \qlb\big)_{{}^*\mathrm{R}_{K_{x_1}}^I \, \chi} [-2r](r).$}

\noindent In particular, if $\chi$ is the trivial character then

\centers{$  \Rgc\big( \X_{\mathbf{x}}, \qlb\big)^{U_J}\, \simeq \, \Rgc\big((\mathbb{G}_m)^d \times  \X_{\L_{J}}(\mathbf{K}_{x_1},\mathbf{v} F), \qlb\big)[-2r](r)$}

\noindent as expected.
\end{rmk}

\subsection{Examples\label{sec33}}

We conclude by showing how Proposition \ref{ypartcase1} can solve the problems encountered in Section \ref{sec24}. As a new application, we determine the contribution of the principal series to the cohomology of a parabolic Deligne-Lusztig variety for a group of type $B_n$. Many other cases will be studied in a subsequent paper. 

\mk

\noindent \thesubsection.1. \textbf{$n$-th roots of $\boldsymbol \pi$ for groups of type $A_n$.} Recall from  \hyperref[sec24]{\ref*{sec24}.3} that for $w = t_1 t_2 \cdots t_n t_{n-1} t_n$ one could decompose the variety $ \X(w)$ into two pieces $\X_{x_n}$ and $\X_{x_1}$ with $x_n = t_n$ and $x_1 = t_n \cdots t_1$ .  However, one could not direcly express the cohomology of the latter. Since $x (t_1 w)x^{-1} = t_1 \cdots t_{n-2} t_{n-1} t_{n-2} \in W_J$ one can now apply Proposition \ref{ypartcase1} to obtain 

\centers{$  \Rgc\big(U_J \backslash  \X_{x_1} , \qlb\big)\, \simeq \, \Rgc\big(\mathbb{G}_m \times  \X_{\L_{J}}(t_1 \cdots t_{n-2} t_{n-1} t_{n-2}), \qlb\big)$}

\noindent in $D^b(\qlb \L_J \times \langle F \rangle$-$\mathrm{mod})$.

\mk

 \noindent \thesubsection.2. \textbf{A new example in type $B_n$.} Let $\G$ be a group of type $B_n$. We denote by $t_1, \ldots, t_n$ the simple reflections of $W$, with the convention that $t_2, \ldots,t_n$ generate a parabolic subgroup of type $A_{n-1}$.  We will restrict our attention to the principal series of $\mathrm{Irr}\ G$, which is parametrised by the representations of the Weyl group $W$. Following \cite{Car}, we will denote by ${[\lambda;\mu]}$ the unipotent character associated to the bipartition $(\lambda,\mu)$ of $n$, with the convention that $\mathrm{Id}_G = [n;-]$ and $\mathrm{St}_G = {[-;1^n]}$.
 
 \sk

 For $n \geq 2$, we consider $w_n = t_n \cdots t_2t_1 t_2 $. It is an $I$-reduced element which normalises $I$ for $I = \{t_1\}$. Then one can use the previous method to determine the principal part of the cohomology of $\X(I,w_n)$, with coefficients in the trivial local system $\qlb$ or in the local system $\mathbf{St}$ associated to the Steinberg representation of $\L_I^{w_nF}$:
 
 \begin{prop} For $n \geq 2$, the contribution of the principal series to the cohomology of $\X(I,w_n)$ with coefficients in  $\qlb$ or $\mathbf{St}$, together with the eigenvalues of $F$, is given by

 \centers{$   \Hc^{n+k}\big(\X(I,w_n), \qlb\big)_{\mathrm{pr}} \, = \, \left\{ \begin{array}{l}  q^{k}\, {[k-1;21^{n-k-1}]} \ \, \text{if } 1 \leq k \leq n-1 \\[4pt]
 q^{n+1}\,  {[n;-]} \ \ \   \text{if } k=n+2 \\[4pt]
 0  \ \  \text{otherwise} \end{array} \right.
$}  

  \leftcenters{and}{$   \Hc^{n+k}\big(\X(I,w_n), \mathbf{St}\big)_{\mathrm{pr}} \, = \, \left\{ \begin{array}{l}  {[-;1^n]} \ \ \   \text{if } k=1 \\[4pt]
  q^{k}\, [(k-1,1);1^{n-k}] \ \, \text{if } 2 \leq k \leq n \\[4pt]
  0  \ \  \text{otherwise.} \end{array} \right.
$}   
 
 \end{prop}

\begin{proof} We proceed by induction. When $n = 2$, \cite[Corollary 8.27]{DM3} applied to $v=1$ forces  $\Hc^\bullet(\X(I,w_2),\qlb)$ and $\Hc^\bullet(\X(I,w_2),\mathbf{St})$ to be $G\times \langle F \rangle$-submodules of $\Hc^\bullet(\X(w_2),\qlb)$. By \cite[Theorem 4.3.4]{DMR}, the latter is multiplicity-free and hence the theorem can be deduced from \cite[Corollary 8.41]{DM3}.

\sk

Assume that $n > 2$ and let $J = \{t_1, \ldots,t_{n-1}\}$. We want to compute the cohomology of $U_J \backslash \X(I,w_n)$. We first observe that any $J$-reduced-$I$ element of $W$ is either $x_i = t_n \cdots t_i$  or $y_i = t_n \cdots t_2 t_1 t_2 \cdots t_i$ for $i>1$. We claim that $\X_{x_i}$ and $\X_{y_j}$ are empty if $i \neq 2$ and $j \neq n$. Indeed, if $i > 2$ then 

\centers{$ W_J^{x_i} = W_J^{y_{i-1}} = \, \langle t_1 ,t_2, \ldots, t_{i-2}, t_{i-1} t_i t_{i-1}, t_{i+1}, \ldots, t_n\rangle.$}

\noindent If $\gamma$ is an $x_i$-distinguished subexpression of an element of $W_I w$ (that is, either $t_1 w$ or $w$) then the product of $\gamma$ is never in $W_J^{x_i}$. Otherwise $\gamma$ would contain neither $t_{i-1}$ nor $t_i$ which is impossible since $\gamma$ is distinguished. The case of $y_{i-1}$ is similar. We deduce that $\X(I,w_n) = \X_{x_2} \amalg \X_{y_n}$. Let us examine each of these two varieties: 

\begin{itemize}

\item we have $x_2 w_n x_2^{-1} = t_{n-1} \cdots t_2 t_1 \in W_J$ and $K_{x_2} = J \cap {}^{x_2} (\Phi_I) = \emptyset$. We can therefore apply Proposition \ref{ypartcase2} and Remark \ref{rmkunip} to obtain

\centers[1]{$ \Rgc(U_J \backslash \X_{x_2},\qlb) \, \simeq \, \Rgc(\X_{\L_J}( t_{n-1} \cdots t_2 t_1),\qlb)[-1]$ }

\leftcenters[1]{and}{$ \Rgc(U_J \backslash \X_{x_2},\mathbf{St}) \, \simeq \, \Rgc(\X_{\L_J}( t_{n-1} \cdots t_2 t_1),\qlb)[-1]$ }

\noindent since the Harish-Chandra restriction of $\mathrm{St}_{\L_I^{w_nF}}$ to $\T^{w_n F}$ is just the trivial character.

\item $y_n = w_0 w_J$ acts trivially on $W_J$ and $y_n (t_n w_n) y_n^{-1} = w_{n-1}$. We have also $K_{y_n} = J \cap {}^{y_n} (\Phi_I) = I$. The assumptions of Proposition \ref{ypartcase1} are clearly satisfied and we obtain 

\centers[1]{$ \Rgc(U_J \backslash \X_{y_n},\qlb) \, \simeq \, \Rgc(\mathbb{G}_m \times \X_{\L_J}(I,w_{n-1}),\qlb)$ }

\leftcenters[1]{and}{$ \Rgc(U_J \backslash \X_{y_n},\mathbf{St}) \, \simeq \, \Rgc(\mathbb{G}_m \times \X_{\L_J}(I,w_{n-1}),\mathbf{St}).$ }

\end{itemize}

\noindent The cohomology of   $\X_{\L_J}( t_{n-1} \cdots t_2 t_1)$ has been computed in \cite{Lu}. By induction, one can assume that the cohomology of $\X_{\L_J}(I,w_{n-1})$ is given by the Theorem (since the unipotent part of the cohomology depends only on the isogeny class of the group). We observe that a character in the principal series different from $\mathrm{Id}$ or $\mathrm{St}$ cannot appear in both $\Hc^\bullet(\X_{\L_J}( t_{n-1} \cdots t_2 t_1),\qlb)$ and $\Hc^\bullet(\X_{\L_J}(I,w_{n-1},\qlb)$ (resp. $\Hc^\bullet(\X_{\L_J}( t_{n-1} \cdots t_2 t_1),\qlb)$ and $\Hc^\bullet(\X_{\L_J}(I,w_{n-1},\mathbf{St})$). Using the long exact sequences given by the decomposition of $U_J \backslash \X(I,w_n)$ and \cite[Corollary 8.28.(v)]{DM3}, we can deduce explicitely each cohomology group of $U_J \backslash \X(I,w_n)$. To conclude, we observe that each of these cohomology groups is the Harish-Chandra restriction of the groups given in the theorem, corresponding to the characters of the principal series in the $\Phi_{2n-2}$-blocks of $\mathrm{Id}_G$ and $\mathrm{St}_G$. Finally, we know by \cite{BMM} that these characters actually appear in the cohomology of $\X(I,w_n)$ since they already appear in the alternating sum. 
\end{proof}

\begin{rmk} In order to deal with the series corresponding to the cuspidal unipotent character of $B_2$ we need extra information on the degree in which $B_{2,\mathrm{Id}}$ and $B_{2,\mathrm{St}}$ can appear.
\end{rmk}

\begin{center} {\subsection*{Acknowledgements}} \end{center}

I would like to thank François Digne and Jean Michel for stimulating discussions while this paper was written. An early version of this work was part of my Ph.D. thesis; I wish to thank Cédric Bonnafé for  many valuable comments and suggestions. 

\bk

\bibliographystyle{abbrv}
\bibliography{dlquotient}

\end{document}